\documentclass[a4paper]{scrartcl}

\usepackage[utf8]{inputenc}
\usepackage[T1]{fontenc} 

\usepackage{amsmath, amsthm, amssymb}
\usepackage{mathrsfs} 
\usepackage{enumitem}
\usepackage{graphicx}
\newlist{hypothenum}{enumerate}{3}
\setlist[hypothenum,1]{label=(\roman*)}
\usepackage{url} 
\usepackage{needspace} 
\usepackage{tikz-cd} 
\usepackage[all,cmtip]{xy} 
\usepackage{graphics} 
\usepackage{colonequals} 
\newcommand*{\defiff}{\ratio\Longleftrightarrow}

\theoremstyle{plain}
\newtheorem*{theorem*}{Theorem}
\newtheorem{theorem}{Theorem}[subsection]
\newtheorem*{proposition*}{Proposition}

\newtheorem*{lemma*}{Lemma}
\newtheorem{lemma}[theorem]{Lemma}
\newtheorem*{corollary*}{Corollary}
\newtheorem{corollary}[theorem]{Corollary}
\newtheorem{corollary_global}{Corollary}
\newtheorem*{claim*}{Claim}

\theoremstyle{definition}
\newtheorem*{definition*}{Definition}

\newtheorem*{example*}{Example}

\newtheorem{example_global}{Example}

\theoremstyle{remark}
\newtheorem*{remark*}{Remark}
\newtheorem*{remarks*}{Remarks}

\numberwithin{equation}{section}

\newcommand{\eps}{\varepsilon}
\newcommand{\setsuch}[2]{\left\{ #1 \; \middle| \; #2 \right\}} 
\newcommand{\restr}[2]{{\left. #1 \right|}_{#2}} 
\newcommand{\ext}{\mathsf{\Lambda}}
\DeclareMathOperator{\rank}{rank}
\DeclareMathOperator{\characteristic}{char}
\DeclareMathOperator{\Ker}{Ker}
\DeclareMathOperator{\Ad}{Ad}
\DeclareMathOperator{\ad}{ad}
\DeclareMathOperator{\Hom}{Hom}
\DeclareMathOperator{\End}{End}
\DeclareMathOperator{\GL}{GL}
\DeclareMathOperator{\SL}{SL}
\DeclareMathOperator{\SO}{SO}

\DeclareMathOperator{\Unit}{U}
\DeclareMathOperator{\Sp}{Sp}
\newcommand{\fundef}[5]{
\entrymodifiers={+!!<0pt,\fontdimen22\textfont2>}
\xymatrix@R=3pt{\llap{$#1$\;\;} {#2} \ar@{->}[r] & {#3} \\ {#4} \ar@{|->}[r] & {#5}}
} 
\newcommand{\mathbffrakg}{\scalebox{1.25}{\ensuremath{\mathfrak{g}}}}
\newcommand{\mathbffraka}{\scalebox{1.25}{\ensuremath{\mathfrak{a}}}}

\newcommand{\ie}{i.e.\ }

\begin{document}

\title{Proper actions on reductive homogeneous spaces}
\author{Yves Benoist\footnote{Translated from French by Ilia Smilga. Original title: Actions propres sur les espaces homogènes réductifs, {\em Ann. of Math.}, 144:315--347, 1996.}}
\date{\vspace{-5ex}} 
\maketitle

\begin{abstract}
Let $G$ be a linear semisimple real Lie group and $H$~be a reductive subgroup of~$G$. We give a necessary and sufficient condition for the existence of a nonabelian free discrete subgroup~$\Gamma$ of~$G$ acting properly on~$G/H$. For instance, such a group~$\Gamma$ does exist for $\SL(2n, \mathbb{R})/\SL(2n-1, \mathbb{R})$ but does not for $\SL(2n+1, \mathbb{R})/\SL(2n, \mathbb{R})$ with~$n \geq 1$.
\end{abstract}

\section{Introduction}
\label{sec:1}

In this introduction, we state our results for the base field~$k = \mathbb{R}$. Later in the text we shall prove analogous results for any local field~$k$.

\subsection{Proper actions of free groups}
\label{sec:1.1}

Let $G$~be a connected linear semisimple real Lie group, $H$~a reductive connected closed subgroup in~$G$ (\ie such that the adjoint action of~$H$ in the Lie algebra of~$G$ is semisimple). We know that $G$ contains an infinite discrete subgroup~$\Gamma$ acting properly on~$G/H$ if and only if $\rank_\mathbb{R}(G) \neq \rank_\mathbb{R}(H)$: this is the Calabi-Markus phenomenon (\cite{Kob1}).

The main goal of this paper is to give a necessary and sufficient condition for existence of a nonabelian free discrete subgroup~$\Gamma$ of~$G$ that acts properly on~$G/H$.

To state this condition, we need a few notations. Let $A_H$~be a Cartan subspace of~$H$ (\ie a connected abelian subgroup composed of hyperbolic elements and maximal for these properties), $A$~a~Cartan subspace of~$G$ containing~$A_H$, $\Sigma := \Sigma(A, G)$ the restricted root system of~$A$ in~$G$, $W$~the Weyl group of~$\Sigma$, $\Sigma^+$~a choice of positive roots, $A^+ := \setsuch{a \in A}{\forall \chi \in \Sigma^+,\; \chi(a) \geq 1}$ the corresponding closed Weyl chamber, $\iota$~the opposition involution for~$A^+$ (for $a$ in~$A^+$, $\iota(a)$~is the unique element of~$A^+$ conjugate to~$a^{-1}$) and $B^+ := \setsuch{a \in A^+}{a = \iota(a)}$.

Note that $B^+$ differs from~$A^+$ if and only if one of the connected components of the Dynkin diagram of~$\Sigma$ is of type $A_n$ with~$n \geq 2$, $D_{2n+1}$ with~$n \geq 1$ or~$E_6$.

We say that a group~$\Gamma$ is virtually abelian if it contains a finite-index abelian subgroup.

\begin{theorem*}
With these notations, $G$ contains a non virtually abelian discrete subgroup~$\Gamma$ acting properly on~$G/H$ if and only if for every~$w \in W$, $w A_H$ does not contain~$B^+$.

In this case, we can choose~$\Gamma$ to be free and Zariski-dense in~$G$.
\end{theorem*}

\begin{example_global}
\label{example_1.1.1}
Here are thus some homogeneous spaces for which such a subgroup~$\Gamma$ does not exist:
\begin{itemize}
\item $\SL(n, \mathbb{R})/(\SL(p, \mathbb{R}) \times \SL(n-p, \mathbb{R}))$ where~$1 \leq p < n$ and $p(n-p)$~is even;
\item $\SL(2p, \mathbb{R})/\SL(p, \mathbb{R})$, $\SL(2p, \mathbb{R})/\SO(p,p)$ and~$\SL(2p+1,\mathbb{R})/\SO(p, p+1)$ where ${p \geq 1}$;
\item $\SO(p+1,q)/\SO(p,q)$ when $p \geq q$ or when $p = q-1$ is even;
\item $G_\mathbb{C}/H_\mathbb{C}$ where $G_\mathbb{C}$~is a complex simple Lie group and $H_\mathbb{C}$~is the set of fixed points of a complex involution of~$G_\mathbb{C}$, except for ${\SO(4n, \mathbb{C})/(\SO(p, \mathbb{C}) \times \SO(4n-p, \mathbb{C}))}$ for $n \geq 2$ and $p$~odd.
\end{itemize}
\end{example_global}

\begin{example_global}
\label{example_1.1.2}
Here are now some homogeneous spaces for which such a subgroup~$\Gamma$ does exist:
\begin{itemize}
\item $\SL(n, \mathbb{R})/(\SL(p, \mathbb{R}) \times \SL(n-p, \mathbb{R}))$ where~$1 \leq p < n$ and $p(n-p)$~is odd;
\item $\SL(n, \mathbb{R})/\SO(p, n-p)$ where $1 \leq p < \lfloor \frac{n}{2} \rfloor$;
\item $\SO(p+1,q)/\SO(p,q)$ when $0 \leq p \leq q-2$ or when $p = q-1$ is odd;
\item ${\SO(4n, \mathbb{C})/(\SO(p, \mathbb{C}) \times \SO(4n-p, \mathbb{C}))}$ for $n \geq 2$ and $p$~odd.
\end{itemize}
\end{example_global}

\subsection{Compact quotients}
\label{sec:1.2}

We say that a homogeneous space~$G/H$ admits a compact quotient if there exists a discrete subgroup~$\Gamma$ of~$G$ acting properly on~$G/H$ and such that the quotient $\Gamma \backslash G / H$ is compact.

Determining whether a given homogeneous space admits a compact quotient is a question for which only partial answers are known (see \cite{Ku}, \cite{Kob1}, \cite{K-O},  \cite{Kob2}, \cite{B-L1} and~\cite{Zi}).

\begin{corollary_global}
\label{corollary_1.2.1}
Keep the same notations, and assume that $G/H$ is not compact and that for a suitable choice of~$\Sigma^+$, $A_H$ contains~$B^+$.

Then $G/H$ does not have a compact quotient.

In particular, none of the homogeneous spaces from Example~\ref{example_1.1.1} have a compact quotient.
\end{corollary_global}

Among these examples, the simplest homogeneous space for which the answer was not already known is $\SL(3, \mathbb{R})/\SL(2, \mathbb{R})$ (see for example Question~3 from the introduction to~\cite{Zi}). Very recently and independently, G. Margulis also proved that the example~$\SL(2n+1,\mathbb{R})/\SL(2n,\mathbb{R})$ has no compact quotient (personal communication).

\begin{corollary_global}
\label{corollary_1.2.2}
Let $G_\mathbb{C}$~be a complex simple Lie group, $\sigma$~a~complex involution of~$G_\mathbb{C}$ and $H_\mathbb{C} := \setsuch{g \in G_\mathbb{C}}{\sigma(g) = g}$. Then the symmetric space~$G_\mathbb{C}/H_\mathbb{C}$ has no compact quotient, except possibly for the case where $G_\mathbb{C}/H_\mathbb{C} = \SO(4n, \mathbb{C})/\SO(4n-1, \mathbb{C})$ for $n \geq 2$.
\end{corollary_global}

For previous results going in this direction, see (\cite{Kob2}~1.9).

~

Here is a more geometric way to state Corollary~\ref{corollary_1.2.1} for the homogeneous space $S^{p, q} := \SO(p+1,q)/\SO(p,q)$.

\begin{corollary_global}
\label{corollary_1.2.3}
No complete compact pseudoriemannian manifold~$V$ of signature~$(p,q)$ with constant sectional curvature~$+1$ exists for $p = 2n$ and~$q = 2n+1$ with $n \geq 1$.
\end{corollary_global}
Indeed such a manifold~$V$ would be a compact quotient of~$S^{p,q}$.

Furthermore it is known (\cite{C-M}, \cite{Wo}, \cite{Ku}) that no such manifolds exist when~$p \geq q$ (the Calabi-Markus phenomenon) or when $pq$~is odd (because of the Gauss-Bonnet formula).

When $p = 1$ and~$q = 2n$ (resp. when $p = 3$ and~$q = 4n$) with~$n \geq 1$, there is an abundant supply of such manifolds~$V$ as the group~$\Unit(1,n)$ (resp. $\Sp(1,n)$) acts properly and transitively on~$S^{p,q}$.

~

Let us now describe the four main steps in the proof of the theorem and of its corollaries.

\subsection{The Cartan projection of~$\Gamma$ {\normalsize \itshape (see Chapter~\ref{sec:3})}}
\label{sec:1.3}

Let $K$~be a maximal compact subgroup of~$G$. We have the Cartan decomposition $G = K A^+ K$ and the Cartan projection $\mu: G \to A^+$: for~$g \in G$, $\mu(g)$~is the unique element of~$K g K \cap A^+$ (see \cite{He}~Ch.~9). For example if $G = \SL(n, \mathbb{R})$, we may take $K = \SO(n)$ and $A^+ = \setsuch{\operatorname{diag}(\sigma_1, \ldots, \sigma_n) \in G}{\sigma_1 \geq \cdots \geq \sigma_n > 0}$; we then have $\mu(g) = \operatorname{diag}(\sigma_1(g), \ldots, \sigma_n(g))$, where $\sigma_i(g)^2$~is the $i$-th eigenvalue of ${}^t g g$.

The first step consists in studying the set~$\mu(\Gamma)$.

\begin{proposition*}
Let $\Gamma$~be a discrete subgroup of~$G$ that is not virtually abelian. Then there exists a compact subset~$M$ of~$A$ such that the set $\mu(\Gamma) \cap B^+ M$ is infinite. In particular, for every closed subgroup~$H'$ of~$G$ containing~$B^+$, $\Gamma$~does not act properly on~$G/H'$.
\end{proposition*}

We explicitly construct infinitely many elements in~$\mu(\Gamma) \cap B^+ M$: we choose $f$ and~$g$ in~$\Gamma$ in a suitable way and we take the elements $\mu(g^p f g^{-p})$, for $p \geq 1$. To check that these elements work, we estimate the norms of their images in a sufficient number of representations of~$G$.

The ``only if'' part of Theorem~\ref{sec:1.1} is of course a consequence of this proposition.

\subsection{Actions of nilpotent groups {\normalsize \itshape (see Chapter~\ref{sec:4})}}
\label{sec:1.4}

Corollary~\ref{corollary_1.2.1} follows from this proposition and from the following proposition which shows that virtually abelian groups~$\Gamma$ do not provide compact quotients either.

\begin{proposition*}
Let $N$~be a nilpotent subgroup of~$G$. Then the quotient $N \backslash G / H$~is not quasicompact.
\end{proposition*}

The proof of this proposition is based on reduction to the case of real rank~$1$.

\subsection{Properness criterion {\normalsize \itshape (see Chapter~\ref{sec:5})}}
\label{sec:1.5}

Let $H_1$, $H_2$~be two closed subgroups of~$G$. We then prove a criterion for properness of the action of~$H_1$ on~$G/H_2$. This criterion depends only on the subsets $\mu(H_1)$ and~$\mu(H_2)$ and on the commutative group~$A$. It generalizes two known criteria, one due to Kobayashi (\cite{Kob1}) when $H_1$ and~$H_2$ are reductive Lie subgroups and the other due to Friedland (\cite{Fr}) when $G = \SL(n, \mathbb{R})$ and~$H_2 = \SL(p, \mathbb{R}) \times I_{n-p}$. Here it is:

\begin{proposition*}
$H_1$~acts properly on~$G/H_2$ if and only if for every compact subset~$M$ of~$A$, the set $\mu(H_1) \cap \mu(H_2) M$ is compact.
\end{proposition*}

\begin{example*}
A discrete subgroup~$\Gamma$ of~$\SL(2p, \mathbb{R})$ acts properly on~$\SL(2p, \mathbb{R})/\Sp(p, \mathbb{R})$ if and only if for every $R > 0$, the set ${\Gamma_R := \setsuch{g \in \Gamma}{\frac{1}{R} \leq \sigma_i(g) \sigma_{2p+1-i}(g) \leq R,\; \forall i = 1, \ldots, p}}$ is finite.
\end{example*}

The main idea of the proof is to estimate, as in~\ref{sec:1.3}, the Cartan projection~$\mu(g)$ of an element $g \in G$ using the norms of the images of~$g$ in a sufficient number of representations of~$G$.

\subsection{Construction of free groups {\normalsize \itshape (see Chapter~\ref{sec:7})}}
\label{sec:1.6}

In the last step, we construct Zariski-dense free subgroups of~$G$ using ideas from \cite{Ti2}, \cite{Mar1}, \cite{G-M} and~\cite{B-L2}.

We denote by $\log$ the logarithm map that identifies~$A$ with its Lie algebra~$\mathfrak{a}$. We say that a subset~$\Omega$ of~$A^+$ is a convex cone if $\log(\Omega)$ is a convex cone in~$\mathfrak{a}$. For example $A^+$ and~$B^+$ are convex cones in~$A$.

\begin{proposition*}
Let $\Omega$~be a nonempty open convex cone in~$A^+$ that is stable by the opposition involution. Then there exists in~$G$ a discrete subgroup~$\Gamma$ that is free on two generators, Zariski-dense in~$G$ and such that $\mu(\Gamma)$ is contained in $\Omega \cup \{e\}$.
\end{proposition*}

G. Margulis told me that he was also aware of this proposition (personal communication).

The group~$\Gamma$ we construct is ``$\eps$-Schottky'', which means that its image in a sufficient number of representations~$V$ of~$G$ is ``$\eps$-Schottky on~$\mathbb{P}(V)$''. Chapter~\ref{sec:6} is dedicated to defining and studying linear groups that are ``$\eps$-Schottky on~$\mathbb{P}(V)$''.

The ``if'' part of Theorem~\ref{sec:1.1} is a consequence of Propositions \ref{sec:1.5} and~\ref{sec:1.6}.

\subsection{}

The proof of Theorem~\ref{sec:1.1} and of Propositions \ref{sec:1.3}, \ref{sec:1.5} and~\ref{sec:1.6} remains valid without any significant changes on any local field~$k$ (see respectively Theorems \ref{sec:7.5}, \ref{sec:3.3}, \ref{sec:5.2} and~\ref{sec:7.4}).

The proof of Proposition~\ref{sec:1.4} and consequently also that of Corollary~\ref{sec:1} remains valid without any changes on any local field of characteristic~$0$ (see respectively Corollaries \ref{sec:4.1} and~\ref{sec:7.6}). Its adaptation to a local field of nonzero characteristic is also possible, but will not be discussed in this paper.

~

These results were announced in~\cite{Be}.

\section{Preliminaries}
\label{sec:2}

\subsection{Local fields}
\label{sec:2.1}

Let $k$~be a local field, \ie either $\mathbb{R}$ or $\mathbb{C}$ or a finite extension of~$\mathbb{Q}_p$ or of~$\mathbb{F}_p[T^{-1}, T]$ for some integer prime~$p$. Let $|.|$~be a continuous absolute value on~$k$.

When $k = \mathbb{R}$ or~$\mathbb{C}$, we set $k^o := (0, \infty)$, $k^+ := [1, \infty)$ and $k^{++} := (1, \infty)$.

When $k$~is non-Archimedean, we call $\mathcal{O}$ the ring of integers of~$k$, $\mathcal{M}$~the maximal ideal of~$\mathcal{O}$ and we choose a uniformizer, \ie an element~$\pi$ of~$\mathcal{M}^{-1}$ which is not in~$\mathcal{O}$. We then set $k^o := \setsuch{\pi^n}{n \in \mathbb{Z}}$, $k^+ := \setsuch{\pi^n}{n \geq 0}$ and~$k^{++} := \setsuch{\pi^n}{n \geq 1}$.

Let $V$ be a finite-dimensional vector space over~$k$. To every basis $v_1, \ldots, v_n$ of~$V$, we associate norms on~$V$ and on~$\End(V)$ defined, for every $v = \sum_{1 \leq i \leq n} x_i v_i$ in~$V$ and for every $g$ in~$\End(V)$, by
\[\|v\| := \sup_{1 \leq i \leq n} |x_i|
\quad\text{and}\quad
\|g\| := \sup_{v \in V,\; \|v\| = 1} \|g \cdot v\|.\]
Of course two different bases of~$V$ give rise to equivalent norms.

\subsection{Cartan decomposition}
\label{sec:2.2}

For every $k$-group $\mathbf{G}$, we denote by~$G$ or by~$\mathbf{G}_k$ the set of its $k$-points and by~$\mathfrak{g}$ its Lie algebra.

Let $\mathbf{G}$ be a semisimple $k$-group. For example $\mathbf{G} = \mathbf{SL}(n, k)$, $\mathbf{Sp}(n, k)$ or the $\mathbf{Spin}$ group of a nondegenerate quadratic form (if $\characteristic k \neq 2$).

Let $\mathbf{A}$ be a maximal $k$-split torus of~$\mathbf{G}$ and (in accordance with our conventions) $A = \mathbf{A}_k$. The dimension~$r$ of~$\mathbf{A}$ is by definition the $k$-rank of~$\mathbf{G}$: $r = \rank_k(\mathbf{G})$. We denote by $X^*(\mathbf{A})$ the set of characters of~$\mathbf{A}$ (this is a free $\mathbb{Z}$-module of rank~$r$) and we set $E := X^*(\mathbf{A}) \otimes_\mathbb{Z} \mathbb{R}$. We denote by $\Sigma = \Sigma(\mathbf{A}, \mathbf{G})$ the set of roots of~$\mathbf{A}$ in~$\mathbf{G}$, \ie the nontrivial weights of~$\mathbf{A}$ in the adjoint representation of the group~$\mathbf{G}$, also called $k$-roots or restricted roots. Then $\Sigma$~is a root system in~$E$ (\cite{B-T1}~\S~5). We choose a system of positive roots~$\Sigma^+$ and we set
\begin{align*}
A^o &:= \setsuch{a \in A}{\forall \chi \in X^*(\mathbf{A}),\; \chi(a) \in k^o}; \\
A^+ &:= \setsuch{a \in A^o}{\forall \chi \in \Sigma^+,\; \chi(a) \in k^+}; \\
A^{++} &:= \setsuch{a \in A^o}{\forall \chi \in \Sigma^+,\; \chi(a) \in k^{++}}.
\end{align*}

Let $N$~be the normalizer of~$A$ in~$G$, $L$~be the centralizer of~$A$ in~$G$ and $W := N/L$ the small Weyl group of~$G$: it can be identified with the Weyl group of the root system~$\Sigma$. The subset~$A^+$ is called the positive Weyl chamber. We have the equality $A^o = \bigcup_{w \in W} w A^+$.

We now assume~$\mathbf{G}$ to be simply connected; this assumption is innocuous, as we can reduce the problem to this case by standard methods (see \cite{Mar2} I.1.5.5 and~I.2.3.1). There exists a maximal compact subgroup~$K$ of~$G$ such that $N = (N \cap K) \cdot A$. We then have the equality $G = K A^+ K$, called the \emph{Cartan decomposition} of~$G$. It follows that for every~$g$ in~$G$, there exists an element~$\mu(g)$ in~$A^+$ such that $g$~is in $K \mu(g) K$. This element $\mu(g)$ is unique. We call \emph{Cartan projection} this map $\mu: G \to A^+$. It is a continuous and proper map. For all of this, we refer to (\cite{He}~9.1.1) in the Archimedean case and to (\cite{Mac}~2.6.11) in the non-Archimedean case.

Let $w_0$ be the ``longest'' element of the Weyl group relative to~$A^+$: it is the unique element of~$W$ such that for every~$a$ in~$A^+$, we have $w_0(a^{-1})$ in~$A^+$. The map $\iota: A^+ \to A^+$ given by
\[\iota(a) = w_0(a^{-1})\]
is called the \emph{opposition involution}. We then have the formula, for every~$g$ in~$G$:
\[\mu(g^{-1}) = \iota(\mu(g)).\]
Finally let $B^+ := \setsuch{a \in A^+}{\iota(a) = a}$ be the set of fixed points of~$\iota$.

\begin{example*}~
\begin{itemize}
\item If $G = \SL(n, k)$ and $k = \mathbb{R}$ or~$\mathbb{C}$, we may take $K = \setsuch{g \in G}{{}^t\overline{g}g = 1}$ and ${A^+ = \setsuch{
\begin{pmatrix}
\sigma_1 & & 0 \\
& \ddots & \\
0 & & \sigma_n
\end{pmatrix}
\in G}{\forall i,\; \sigma_i \in \mathbb{R} \text{ and } \sigma_1 \geq \cdots \geq \sigma_n > 0}}$.
\item If $G = \SL(n, k)$ and $k$ is non-Archimedean, we may take $K = \SL(n, \mathcal{O})$ and ${A^+ = \setsuch{
\begin{pmatrix}
\pi^{q_1} & & 0 \\
& \ddots & \\
0 & & \pi^{q_n}
\end{pmatrix}
\in G}{\forall i,\; q_i \in \mathbb{Z} \text{ and } q_1 \geq \cdots \geq q_n}}$.
\item In both cases, for $g$ in~$G$, we denote by~$\sigma_i(g)$ the diagonal coefficients of~$\mu(g)$: these are the \emph{singular values} of~$g$. When $k = \mathbb{R}$ or~$\mathbb{C}$, these are the eigenvalues of~$({}^t\overline{g}g)^\frac{1}{2}$.
\end{itemize}
\end{example*}

\subsection{Representations of~$\mathbf{G}$}
\label{sec:2.3}

Though the reminders given in this paragraph are valid on any infinite field~$k$, we keep the notations from~\ref{sec:2.2}. Let $\rho$~be a representation of~$G$ on a finite-dimensional $k$-vector space~$V$. More precisely, $\rho$~is a $k$-morphism of $k$-groups $\rho: \mathbf{G} \to \mathbf{GL}(V)$. For $\chi \in X^*(\mathbf{A})$, we denote by $V_\chi := \setsuch{v \in V}{\forall a \in A,\; \rho(a)v = \chi(a)v}$ the corresponding eigenspace. We denote by $\Sigma(\rho) := \setsuch{\chi \in X^*(\mathbf{A})}{V_\chi \neq 0}$ the set of $k$-weights of~$V$. This set is invariant by the action of the Weyl group~$W$ and we have
\[V = \bigoplus_{\chi \in \Sigma(\rho)} V_\chi.\]
We endow $X^*(\mathbf{A})$ with the order given by
\[\chi_1 \leq \chi_2 \;\defiff\; \chi_2 - \chi_1 \in \sum_{\chi \in \Sigma^+} \mathbb{N} \chi.\]

We assume that $\rho$~is irreducible. The set~$\Sigma(\rho)$ then has a unique element~$\lambda$ that is maximal for this order, called the highest $k$-weight of~$V$. When $\mathbf{G}$ is $k$-split, we have $\dim V_\lambda = 1$.

We will need the following well-known lemma.
\begin{lemma*}
There exist $r$ irreducible representations~$\rho_i$ of the group~$G$ on $k$-vector spaces~$V_i$ whose highest $k$-weights $(\omega_i)_{1 \leq i \leq r}$ form a basis of the $\mathbb{R}$-vector space~$E$ and such that ${\dim (V_i)_{\omega_i} = 1}$.
\end{lemma*}

For complete results concerning classification of representations of~$G$, we refer to~\cite{B-T1, B-T2} as well as to~\cite{Ti1}.

\begin{proof}
We choose irreducible representations~$\sigma_i$ of~$G$ on $k$-vector spaces~$W_i$ whose highest $k$-weights $(\lambda_i)_{1 \leq i \leq r}$ form a basis of the $\mathbb{R}$-vector space~$E$ (\cite{Ti1}~7.2). We set $d_i := \dim (W_i)_{\lambda_i}$, $\omega_i := d_i \lambda_i$ and we take for~$V_i$ the simple subquotient of $\ext^{d_i}(W_i)$ having $\omega_i$ as a $k$-weight.
\end{proof}

\begin{example*}
When $G = \SL(n, k)$, we have $r = n-1$ and $V_i = \ext^i (k^n)$, for $1 \leq i \leq n-1$.
\end{example*}

\subsection{Cartan projection and representations of~$\mathbf{G}$}
\label{sec:2.4}

The following lemma is easy and fundamental. In the light of Lemma~\ref{sec:2.3}, it says that, up to a bounded multiplicative constant, computing~$\mu(g)$ is equivalent to computing the norms~$\|\rho_i(g)\|$ for $i = 1, \ldots, r$.

\begin{lemma*}
For every irreducible representation $(V, \rho)$ of~$G$ with highest $k$-weight $\chi$ and for every norm on~$V$, there exists a constant $C_\chi > 0$ such that, for every $g \in G$, we have
\[C_\chi^{-1} \leq \frac{|\chi(\mu(g))|}{\|\rho(g)\|} \leq C_\chi.\]
\end{lemma*}

\begin{proof}
We may assume that the chosen norm corresponds to a basis formed by eigenvectors for the action of~$A$ (see~\ref{sec:2.1}) so that, for $a$ in~$A^+$, we have
\[|\chi(a)| = \|\rho(a)\|.\]

Take $a = \mu(g)$ so that $g = k_1 a k_2$ with $k_1$, $k_2$ in~$K$. We then have
\[\frac{|\chi(\mu(g))|}{\|\rho(g)\|} = \frac{\|\rho(a)\|}{\|\rho(k_1 a k_2)\|} \in [C_\chi^{-1}, C_\chi]\]
where $C_\chi = \sup_{k \in K}\|\rho(k)\|^2$. This is what we wanted.
\end{proof}

\section{The Cartan projection of~$\Gamma$}
\label{sec:3}

The goal of this section is to prove Theorem~\ref{sec:3.3} which generalizes Proposition~\ref{sec:1.3}.

\subsection{$H$-proper pairs}
\label{sec:3.1}

Let us start with a few easy observations whose power will become apparent later.

\begin{definition*}
Let $H$~be a locally compact group and $H_1$, $H_2$ two closed subsets of~$H$. We shall say that $(H_1, H_2)$ is $H$-proper if for every compact subset~$L$ of~$H$, the intersection $H_1 \cap L H_2 L$ is compact.
\end{definition*}

\begin{remarks*}~
\begin{enumerate}
\item \label{remark_3.1.1} If $(H_1, H_2)$ is $H$-proper, then $(H_2, H_1)$ is $H$-proper and moreover for any $h_1$, $h_2$ in~$H$, $(h_1 H_1 h_1^{-1}, h_2 H_2 h_2^{-1})$ is $H$-proper.
\item \label{remark_3.1.2} This definition may be generalized as follows: let $E$ and~$X$ be locally compact topological spaces and $a: E \times X \to X$ a continuous map, in other terms $a$~is a continuous family $e \mapsto a_e$ of continuous transformations of~$X$. We shall say that this family is proper if for every compact subset~$L$ of~$X$, the set $E_L := \setsuch{e \in E}{e L \cap L \neq \emptyset}$ is compact. When $E$~is a semigroup and $a$~is an action, we get the usual definition of a proper action.

Saying that $(H_1, H_2)$ is $H$-proper is equivalent to saying that the family of transformations of~$H$
\begin{align*}
(H_1 \times H_2) \times H &\to H \\
((h_1, h_2), h) &\mapsto a_{h_1, h_2}(h) := h_1 h h_2^{-1}
\end{align*}
is proper.
\end{enumerate}
\end{remarks*}

The following lemma is a direct application of the definitions. Its verification is left to the reader.

\begin{lemma}
\label{lemma_3.1.1}
Let $H$ be a locally compact group, $H_1$~a closed subsemigroup of~$H$ and $H_2$~a closed subgroup of~$H$. We have the equivalence:

$(H_1, H_2)$ is $H$-proper $\iff$ $H_1$ acts properly on~$G/H_2$.
\end{lemma}

\begin{definition*}
Let $H$ be a locally compact group and $H_1$, $H'_1$ two closed subsets of~$H$. We shall say that $H'_1$~is contained in~$H_1$ modulo the compacts of~$H$ if there exists a compact subset~$L$ of~$H$ such that $H'_1 \subset L H_1 L$.
\end{definition*}

\begin{lemma}
\label{lemma_3.1.2}
Let $H$ be a locally compact group, and let $H_1$, $H'_1$, $H_2$,~$H'_2$ be four closed subsets of~$H$ such that $(H_1, H_2)$ is $H$-proper and $H'_j$~is contained in~$H_j$ modulo the compacts of~$H$ for $j = 1, 2$. Then the pair $(H'_1, H'_2)$ is $H$-proper.
\end{lemma}

This lemma is also an immediate consequence of the definitions. It is the conjunction of this lemma with the Cartan decomposition that explains the Calabi-Markus phenomenon.

\subsection{Growth of~$g_1^p f g_2^p$}
\label{sec:3.2}

Let us use the notations of~\ref{sec:2.2} again. Let $\mathbf{G}$ be a simply-connected semisimple $k$-group, $G := \mathbf{G}_k$ and $\mu: G \to A^+$ a Cartan projection.

\begin{proposition*}
Let $g_1$, $g_2$ be two elements of~$G$ and $F$ a nonempty subset of~$G$. Then there exists a nonempty subset~$F'$ of~$F$ which is Zariski-open in~$F$ and such that for any $f$, $f'$ in~$F'$, there exists a compact subset~$M_{f, f'}$ of~$A$ such that for every $p \geq 1$, we have
\[\mu(g_1^p f g_2^p) \cdot \mu(g_1^p f' g_2^p)^{-1} \in M_{f, f'}.\]
\end{proposition*}

\begin{remark*}
In this statement, the phrase ``Zariski-open in~$F$'' means open with respect to the topology induced on~$F$ by the Zariski topology of~$G$. In other terms, $F'$ is a nonempty subset of~$F$ whose complement is defined by polynomial equations.
\end{remark*}

\begin{proof}
We may assume that $F$~is Zariski-connected so that the intersection of two nonempty Zariski-open subsets of~$F$ is still a nonempty Zariski-open subset of~$F$.

The proposition is then a consequence of Lemma~\ref{sec:2.4} and of the following elementary lemma applied to the subsets $\rho_i(F)$ where the~$\rho_i$ are the representations of~$G$ introduced in Lemma~\ref{sec:2.3}.
\end{proof}

Let $V$ be a $k$-vector space of dimension~$d$. Let us take the notations of~\ref{sec:2.1}.

\begin{lemma*}
Let $g_1, g_2$ be elements of~$\GL(V)$ and let~$F$ be a nonempty subset of $\End(V) \setminus \{0\}$. Then there exists a nonempty subset~$F'$ of~$F$, Zariski-open in~$F$ and such that for any $f$, $f'$ in~$F'$, there exists a constant $C_{f, f'} > 1$ such that for every $p \geq 1$, we have
\[C_{f, f'}^{-1} \leq \|g_1^p f g_2^p\| \|g_1^p f' g_2^p\|^{-1} \leq C_{f, f'}.\]
\end{lemma*}

We suggest to the reader to prove this lemma in the particular case where $g_1$ and~$g_2$ are diagonal matrices before reading the complete proof. 

\begin{proof}
We denote by~$\phi$ the endomorphism of~$\End(V)$ given by $\phi(f) := g_1 f g_2$. Replacing if necessary $k$ by a finite extension, we may assume that the eigenvalues of~$\phi$ are all in~$k$.

We endow the set $(0, \infty) \times \mathbb{N}$ (where $\mathbb{N}$ stands for the set of nonnegative integers) with the lexicographic order:
\[(\lambda, r) \leq (\lambda', r') \;\defiff\; \lambda < \lambda' \;\text{ or }\; (\lambda = \lambda' \text{ and } r \leq r').\]
We introduce, for $\lambda$ in~$(0, \infty)$ and $r$ in~$\mathbb{N}$, the following vector subspace of~$\End(V)$:
\[W^{\lambda, r} := \sum_{\substack{(z, s) \in k \times \mathbb{N} \\ (|z|, s) < (\lambda, r)}} \Ker ((\phi - z)^s).\]
Let $(\lambda, r)$ be the greatest element of~$(0, \infty) \times \mathbb{N}$ such that $F \not\subset W^{\lambda, r}$ and let $F' := F \setminus (F \cap W^{\lambda, r})$. It is clear that for every~$f$ in~$F'$, there exists a constant $A_f > 0$ such that the sequence $p \mapsto \|\phi^p(f)\| = \|g_1^p f g_2^p\|$ is equivalent to~$A_f p^r \lambda^p$. The conclusion follows.
\end{proof}

\subsection{Fixed points of the opposition involution}
\label{sec:3.3}

\begin{theorem*}
Let $k$ be a local field, $\mathbf{G}$~a semisimple $k$-group, $\mathbf{A}$~a maximal $k$-split torus of~$\mathbf{G}$, $A^+$~a positive Weyl chamber, $\iota$~the opposition involution, $B^+ := \setsuch{a \in A^+}{\iota(a) = a}$ and $\Gamma$~a discrete subgroup of~$G := \mathbf{G}_k$. We assume $\characteristic k = 0$ (resp. $\characteristic k \neq 0$).

If the pair $(\Gamma, B^+)$ is $G$-proper then $\Gamma$~is virtually abelian (resp. nilpotent).

In particular, if $H$ is a subgroup of~$G$ containing~$B^+$ and if $\Gamma$~acts properly on~$G/H$ then $\Gamma$~is virtually abelian (resp. nilpotent).
\end{theorem*}

\begin{proof}
We may assume $\mathbf{G}$ to be simply connected and $\Gamma$ to be Zariski-connected. We denote by $\mu: G \to A^+$ a Cartan projection. Let $g$~be an element of~$\Gamma$. By the previous proposition, there exists a nonempty subset $\Gamma'$ of~$\Gamma$ which is Zariski-open in~$\Gamma$ and such that for every $f$ in~$\Gamma'$, there exists a compact subset~$M_f$ of~$A$ such that, for every $p \geq 1$,
\[\mu(g^p f g^{-p}) \cdot \mu(g^p f^{-1} g^{-p})^{-1} \in M_f.\]
Hence by the lemma below, there exists a compact subset $M'_f$ in~$A$ such that, for every $p \geq 1$,
\[\mu(g^p f g^{-p}) \in M'_f B^+.\]

By assumption, the set $\mu(\Gamma) \cap M'_f B^+$ is compact. Since $\Gamma$~is discrete and $\mu$~is proper, the set $\setsuch{g^p f g^{-p}}{p \geq 0}$ is finite. Let $Z_f$~be the centralizer of~$f$ in~$\Gamma$ and $Z_\Gamma$~be the center of~$\Gamma$. Hence there exists $p_0 \geq 1$ such that $g^{p_0}$~is in~$Z_f$.

By Noetherianness, there exists a finite subset~$\Gamma_0$ of~$\Gamma$ such that
\[Z_\Gamma = \bigcap_{\gamma \in \Gamma_0} Z_\gamma.\]
Since $\Gamma'$ generates~$\Gamma$, we may assume that $\Gamma_0$ is contained in~$\Gamma'$. Hence there exists $p \geq 1$ such that $g^p$~is in~$Z_\Gamma$.

The group $\Gamma / Z_\Gamma$ is a linear torsion group. The claim below shows that $\Gamma / Z_\Gamma$ contains a finite-index nilpotent subgroup. Since $\Gamma$~is Zariski-connected, $\Gamma$~is nilpotent.

Now if $k$ has zero characteristic, since $\Gamma$~is nilpotent and discrete, $\Gamma$~is finitely generated. But then $\Gamma / Z_\Gamma$ is a finite group. This is what we wanted.
\end{proof}

In this proof, we used the following lemma and claim.
\begin{lemma*}
Let $M$ be a compact subset of~$A$. Then there exists a compact subset~$M'$ of~$A$ such that, for every $a$ in~$A^+$, we have the implication
\[a \cdot \iota(a)^{-1} \in M \implies a \in B^+ M'.\]
\end{lemma*}

\begin{proof}
When $k = \mathbb{R}$, the logarithm map identifies the connected component~$A_e$ of~$A$ to an $\mathbb{R}$-vector space, the involution~$\iota$ to a linear symmetry and $A^+$~to a convex cone invariant by~$\iota$. We may then take $M' := \setsuch{m \in A}{m^2 \in M}$.

The general case is no harder. There exists a finite subset~$L$ of~$A^+$ such that for every~$a$ in~$A^+$, there exists~$l$ in~$L$ and $c$ in~$A^+$ such that $a = c^2 l$. We may then write $a = b m$ where $b := c \cdot \iota(c)$~is in~$B^+$ and $m := l \cdot c \cdot \iota(c)^{-1}$~is in the compact set $M' := \setsuch{m \in A}{m^2 \in M L \iota(L)^{-1}}$.
\end{proof}

\begin{claim*}
Let $k$~be a field, $V$ a finite-dimensional $k$-vector space and $\Gamma$ a torsion subgroup of~$\GL(V)$.
\begin{enumerate}[label=\alph*)]
\item (Schur, see \cite{C-R} p.~258) If $\characteristic(k) = 0$, $\Gamma$~contains a finite-index abelien subgroup whose elements are all semisimple.
\item (\cite{Ti2}~2.8) If $\characteristic(k) \neq 0$, we denote by~$k_a$ the algebraic closure in~$k$ of the prime subfield of~$k$. Then every simple subquotient~$V'$ of~$V$ has a basis in which the coefficients of the elements of~$\Gamma$ lie in~$k_a$. In particular, if $k_a$~is finite (which is the case when $k$~is local), $\Gamma$ has a finite-index nilpotent subgroup whose elements are all unipotent.
\item (Burnside) In both cases, if $\Gamma$~is finitely generated, then $\Gamma$~is finite.
\end{enumerate}
\end{claim*}

\section{Action of nilpotent groups}
\label{sec:4}
In this section $k$~is a local field of characteristic zero. The goal of this section is to prove Theorem~\ref{sec:4.1} and its corollary which generalizes Proposition~\ref{sec:1.4}.

\subsection{Non-quasicompactness of~$N \backslash G / H$}
\label{sec:4.1}
\begin{theorem*}
Let $k$ be a local field of characteristic zero, $\mathbf{G}$ a reductive $k$-group, $\mathbf{H}$ a $k$-subgroup of~$\mathbf{G}$, $G := \mathbf{G}_k$, $H := \mathbf{H}_k$ and $N$ a nilpotent subgroup of~$G$.

If $N \backslash G / H$ is quasicompact, then $\mathbf{H}$ contains a maximal unipotent $k$-subgroup of~$\mathbf{G}$.
\end{theorem*}

The proof of this proposition is based on a reduction to the case of $k$-rank one. It is done in sections \ref{sec:4.2} to~\ref{sec:4.5}. First of all let us state a corollary of this theorem.

\begin{corollary*}
Same notations. We assume that $G/H$ is not compact.
\begin{enumerate}[label=\alph*)]
\item If $\mathbf{H}$~is reductive then $N \backslash G / H$ is not quasicompact.
\item ($k = \mathbb{R}$) If $N$ acts properly on~$G/H$ then $N \backslash G / H$ is not compact.
\end{enumerate}
\end{corollary*}

\begin{proof}~
\begin{enumerate}[label=\alph*)]
\item We may assume that $\mathbf{G}$ is simply connected. We decompose $\mathbf{G}$ into a product $\mathbf{G} = \mathbf{G}_{\mathrm{an}} \times \mathbf{G}_{\mathrm{is}}$ where $\mathbf{G}_{\mathrm{an}}$~is the largest anisotropic connected normal $k$-subgroup of~$\mathbf{G}$ and $\mathbf{G}_{\mathrm{is}}$~is the largest connected normal $k$-subgroup of~$\mathbf{G}$ that has no anisotropic factor.

Otherwise $\mathbf{H}$ contains a maximal unipotent $k$-subgroup of~$\mathbf{G}$. Since $\mathbf{H}$~is reductive, $\mathbf{H}$ contains~$\mathbf{G}_{\mathrm{is}}$ hence $G/H$~is a quotient of the group of $k$-points $G_{\mathrm{an}}$ which is compact (\cite{B-T1}~9.4). Contradiction.
\item Otherwise $\mathbf{H}$ contains a maximal unipotent $k$-subgroup~$\mathbf{U}$ of~$\mathbf{G}$. There exists a maximal compact subgroup~$K$ of~$G$ such that $G = K U K$ (\cite{Kos}~5.1). It follows that $G = K H K$ and $G$~is contained in~$H$ modulo the compacts of~$G$. Lemma~\ref{lemma_3.1.1} then shows that the pair~$(N, G)$ is $G$-proper. Hence $N$~is a compact group. Contradiction. \qedhere
\end{enumerate}
\end{proof}

\subsection{Parabolic $k$-subgroups}
\label{sec:4.2}

Let us introduce a few classical notations (see \cite{Bor}~\S 21) that will be useful in this section only.

Let $\mathbf{A}$ be a maximal $k$-split torus of~$\mathbf{G}$, $\Sigma = \Sigma_\mathbf{G}$ the system of the $k$-roots of~$\mathbf{A}$ in~$\mathbf{G}$, $\Sigma^+$~a choice of positive roots, $\Pi$~the simple $k$-roots of~$\Sigma^+$, $\mathbf{L}$~the centralizer of~$\mathbf{A}$ in~$\mathbf{G}$ and $\mathbffrakg$ the Lie algebra of~$\mathbf{G}$. As usual, we denote by the corresponding Roman letter $A$, $L$ etc. the group of $k$-points.

For $\alpha$ in~$\Sigma$, we denote by $\mathbffrakg_\alpha := \setsuch{X \in \mathbffrakg}{\forall a \in A,\; \operatorname{Ad} a(X) = \alpha(a)X}$ the corresponding root space, $\mathbffrakg_{(\alpha)} := \mathbffrakg_\alpha \oplus \mathbffrakg_{2\alpha}$ and $\mathbf{U}_{(\alpha)}$ the unique unipotent $k$-subgroup (normalized by~$\mathbf{L}$) with Lie algebra~$\mathbffrakg_{(\alpha)}$. We say that a subset~$\Theta$ of~$\Sigma$ is closed if ${\alpha, \beta \in \Theta,}\; {\alpha + \beta \in \Sigma} \implies {\alpha + \beta \in \Theta}$. We denote by $\langle \Theta \rangle$ the smallest closed subset of~$\Sigma$ containing~$\Theta$. For every closed subset~$\Theta$ of~$\Sigma^+$, we call $\mathbf{U}_\Theta$ the unique unipotent $k$-subgroup (normalized by~$\mathbf{L}$) with Lie algebra $\mathbffrakg_\Theta := \bigoplus_{\alpha \in \Theta} \mathbffrakg_\alpha$. We set $\mathbf{U} := \mathbf{U}_{\Sigma^+}$, this is a maximal unipotent $k$-subgroup of~$\mathbf{G}$. The $k$-group $\mathbf{P} := \mathbf{L}\mathbf{U}$ is a minimal parabolic $k$-subgroup of~$\mathbf{G}$. For every subset $\Theta$ of~$\Pi$, we denote by~$\mathbf{A}^\Theta$ the Zariski-connected component of~$\bigcap_{\alpha \in \Theta} \mathbf{Ker}(\alpha)$, $\mathbf{L}_\Theta$~the centralizer of~$A^\Theta$, $\mathbf{U}^\Theta := \mathbf{U}_{\Sigma^+ \setminus \langle \Theta \rangle}$ and $\mathbf{P}_\Theta := \mathbf{L}_\Theta \mathbf{U}^\Theta$~the standard parabolic $k$-subgroup associated to~$\Theta$. We also have the equality for the $k$-points $P_\Theta = L_\Theta U^\Theta$ (\cite{B-T1}~3.14).

In the following well-known lemma, we do not need to assume that $k$~is a local field.

\begin{lemma}
\label{lemma_4.2.1}
($\characteristic(k) = 0$)
\begin{enumerate}[label=\alph*)]
\item Let $\mathbf{U}'$ be a $k$-subgroup of~$\mathbf{U}$ and $[\mathbf{U}, \mathbf{U}]$ the derived subgroup of~$\mathbf{U}$. If we have $\mathbf{U} = {\mathbf{U}' \cdot [\mathbf{U}, \mathbf{U}]}$, then $\mathbf{U}' = \mathbf{U}$.
\item We have $[\mathbf{U}, \mathbf{U}] = \mathbf{U}_{\Sigma^+ \setminus \Pi}$.
\end{enumerate}
\end{lemma}

\begin{proof}~
\begin{enumerate}[label=\alph*)]
\item This holds for any unipotent group: if $\mathbf{U}' \neq \mathbf{U}$, by Engel's theorem, we may assume that $\mathbf{U}'$ has codimension~$1$; $\mathbf{U}'$~is then normal in~$\mathbf{U}$ and $\mathbf{U} / \mathbf{U}'$ is abelian. Hence $\mathbf{U} \neq \mathbf{U}' \cdot [\mathbf{U}, \mathbf{U}]$. Contradiction.
\item This follows from the equality $[\mathbffrakg_\alpha, \mathbffrakg_\beta] = \mathbffrakg_{\alpha + \beta}$ for any $\alpha, \beta$ in~$\Sigma^+$. \qedhere
\end{enumerate}
\end{proof}

The following lemma will allow us to assume that $\mathbf{H}$ is solvable $k$-split.

\begin{lemma}
\label{lemma_4.2.2}
Let $\mathbf{G}$ be a reductive $k$-group, $\mathbf{H}$ a $k$-subgroup. Then there exists a maximal $k$-split torus~$\mathbf{A}$ of~$\mathbf{G}$ and a maximal unipotent $k$-subgroup $\mathbf{U}$ normalized by~$\mathbf{A}$ such that $H/(H \cap A U)$ is compact.
\end{lemma}

In other terms, $H$ meets the maximal $k$-split solvable $k$-subgroup $\mathbf{A}\mathbf{U}$ along a subgroup that is cocompact.

\begin{proof}
Let $\mathbf{A}$~be a maximal $k$-split subtorus of~$\mathbf{G}$ and $\mathbf{U}$~be a maximal unipotent $k$-subgroup normalized by~$\mathbf{A}$. Then by the Iwasawa decomposition (see (\cite{He}~IX.1.3) for the Archimedean case and \cite{Mac} for the non-Archimedean case), the homogeneous space~$G/A U$ is compact.

By~(\cite{B-T2}~3.18), the orbits of~$H$ in~$G/A U$ are locally closed. Hence $H$~has a closed orbit in~$G/A U$. Without loss of generality, it is the orbit of the basepoint. The quotient $H/(H \cap A U)$~is then homeomorphic to this orbit (loc. cit.) and in particular $H/(H \cap A U)$~is compact.
\end{proof}

\subsection{Reduction to the case of $k$-rank one}
\label{sec:4.3}
Let us show that Theorem~\ref{sec:4.1} holds in all generality if it holds for the $k$-groups~$\mathbf{G}$ whose $k$-rank is equal to~$1$.

By~\ref{lemma_4.2.2}, we may assume $\mathbf{H} \subset \mathbf{A}\mathbf{U}$. By (\cite{Bor} 10.6 and~19.2), we may assume that $\mathbf{H} = \mathbf{A}'\mathbf{U}'$ where $\mathbf{A}'$~is a subtorus of~$\mathbf{A}$ and $\mathbf{U}'$~is a unipotent $k$-subgroup of~$\mathbf{U}$. We then also have equality for the $k$-points: $H = A'U'$ (\cite{Bor}~15.8).

Similarly, we may assume that $N$ is the group of $k$-points of its Zariski-closure $\mathbf{N}$, that $\mathbf{N}$ is a maximal Zariski-connected nilpotent $k$-subgroup of~$\mathbf{A}\mathbf{U}$ and that $\mathbf{N} = \mathbf{A}''\mathbf{U}''$ where $\mathbf{A}''$~is a subtorus of~$\mathbf{A}$ and $\mathbf{U}''$~is a unipotent $k$-subgroup of~$\mathbf{U}$. Hence there exists a subset~$\Theta$ of~$\Pi$ such that $\mathbf{A}'' = \mathbf{A}^\Theta$ and $\mathbf{U}'' = \mathbf{U}_{\langle \Theta \rangle}$. We still have equality for the $k$-points: $N = A''U''$.

We want to show that $\mathbf{U}' = \mathbf{U}$. If it is not the case, we may assume, thanks to~\ref{lemma_4.2.1}, that $\mathbf{U}'$ contains $[\mathbf{U}, \mathbf{U}]$ and that $\mathbf{U}'$ has codimension~$1$ in~$\mathbf{U}$. Let us write $\mathbf{U}'_{(\alpha)} := \mathbf{U}' \cap \mathbf{U}_{(\alpha)}$ and $\Xi := \setsuch{\alpha \in \Pi}{\mathbf{U}'_{(\alpha)} \neq \mathbf{U}_{(\alpha)}}$. The set $\Xi$ is nonempty.

Since $\mathbf{U}'$ has codimension~$1$ in~$\mathbf{U}$ and is normalized by~$\mathbf{A}'$, we have $\mathbf{A}' \subset \mathbf{Ker}(\alpha - \alpha')$ for any $\alpha$, $\alpha'$ in~$\Xi$. We may suppose that $\mathbf{A}'$ is the connected component of~$\displaystyle \bigcap_{\alpha, \alpha' \in \Xi} \mathbf{Ker}(\alpha - \alpha')$.

The subgroup $A U$ is closed in~$G$ and contains $N$ and~$H$. Hence $N \backslash A U / A' U'$ is quasicompact and so is $A'' \backslash A / A'$. We deduce that $A''A'$ has finite index in~$A$ (\cite{Bor}~8.5). Hence $\# (\Xi \cap \Theta) \leq 1$. For more clarity, let us distinguish two cases:

~

\underline{1st case:} $\Xi \cap \Theta = \emptyset$. Let us fix an element~$\alpha$ in~$\Xi$. Let $\mathbf{P}_\alpha$ denote the standard parabolic $k$-subgroup associated to~$\{\alpha\}$, and let $\mathbf{P}_\alpha = \mathbf{L}_\alpha \mathbf{U}^\alpha$ be its Levi decomposition. The group of $k$-points~$P_\alpha$ is closed in~$G$ and contains $N$ and~$H$, hence $N \backslash P_\alpha / H$ is quasicompact.

We have the equality $U = U_{(\alpha)} \cdot U'$ since $U/U'$ identifies with a one-dimensional $k$-vector space and the image of~$U_{(\alpha)}$ is a nontrivial $k$-vector subspace of that space. Hence we have the equality $P_\alpha = L_\alpha U = L_\alpha U'$ and the identification
\[P_\alpha / H \simeq L_\alpha / A'U'_{(\alpha)}.\]
Let us study the action of~$N$ on that quotient. On the one hand, the group~$U^\Xi$ is normalized by~$P_\alpha$ and is contained in~$U'$, hence it acts trivially on that quotient. So does its subgroup~$U''$. On the other hand, the group~$A''$ is contained in~$L_\alpha$. Hence
\[N \backslash P_\alpha / H \simeq A'' \backslash L_\alpha / A'U'_{(\alpha)}\]
is not quasicompact since the semisimple $k$-rank of~$L_\alpha$ is equal to~$1$. Contradiction.

~

\underline{2nd case:} $\Xi \cap \Theta = \{\alpha\}$. The quotient $N \backslash P_\alpha / H$ is still quasicompact and we have the equality $P_\alpha = L_\alpha U'$ and hence the same identification
\[P_\alpha / H \simeq L_\alpha / A'U'_{(\alpha)}.\]
This time we have the equality $N = A''U'' = (A''U_{(\alpha)}) \cdot U''^\alpha$ where $U''^\alpha := U'' \cap U^\alpha$. Once again $U''^\alpha$ is contained in~$U^\Xi$ since $\langle \Theta \rangle \cap \langle \Xi \rangle = \langle \alpha \rangle$ and $U''^\alpha$ acts trivially on this quotient. On the other hand, $A''$~is still included in~$L_\alpha$, it is now even included in the center of~$L_\alpha$. Hence
\[N \backslash P_\alpha / H \simeq A''U_{(\alpha)} \backslash L_\alpha / A'U'_{(\alpha)}\]
is not quasicompact for the same reason. Contradiction.

\subsection{The case of $k$-rank one}
\label{sec:4.4}
To finish the proof of Theorem~4.1, we may thus assume that the semisimple $k$-rank of~$\mathbf{G}$ is equal to~$1$. Let $\mathbf{Z}$~be the center of~$\mathbf{G}$; the quotient $N \backslash G / (\mathbf{Z}\mathbf{H})_k$ is still quasicompact. We may thus assume that $\mathbf{G}$~is semisimple and adjoint.

Let $\alpha$~be the unique element of~$\Pi$ and $\mathbf{V} := \mathbf{U}_{(-\alpha)}$. The following lemma is a crucial ingredient of our proof.

\begin{lemma*}
($\characteristic(k) = 0, \rank_k(\mathbf{G}) = 1$)
\begin{enumerate}[label=\alph*)]
\item Let $v_n$, $u_n$ and~$a_n$ be sequences respectively in $V$, $U$ and~$A$ such that the sequence $g_n := v_n u_n a_n$ converges in~$G$. Then the sequence $u_n$ is bounded in~$U$.
\item There exists an element $v_0$ of~$V$ such that if $b_n$, $u_n$ and~$a_n$ are sequences respectively in $A$, $U$ and~$A$ such that the sequence $h_n := b_n v_0 u_n a_n$ converges in~$G$, then the sequence~$u_n$ is bounded in~$U$.
\end{enumerate}
\end{lemma*}

\begin{remarks*}~
\begin{itemize}
\item The condition ``$u_n$ bounded in~$U$'' means that the sequence~$u_n$ remains in a compact subset of~$U$. It is likely that in both cases, the sequence~$u_n$ converges. We shall not need this fact.
\item This lemma is false for $\mathbf{G} = \mathbf{SL}(3, k)$, which shows the necessity of the reduction to the case of $k$-rank one. Indeed, take (with $t = t_n \to 0$):
\[g_n = 
\begin{pmatrix}
1 & 0 & 0 \\
0 & 1 & 0 \\
t^{-1} & -t^{-2} & 1
\end{pmatrix}
\cdot
\begin{pmatrix}
1 & t^{-1} & 0 \\
0 & 1 & t^2 \\
0 & 0 & 1
\end{pmatrix}
\cdot
\begin{pmatrix}
t & 0 & 0 \\
0 & t & 0 \\
0 & 0 & t^{-2}
\end{pmatrix}
=
\begin{pmatrix}
t & 1 & 0 \\
0 & t & 1 \\
1 & 0 & 0
\end{pmatrix}.\]
\end{itemize}
\end{remarks*}

Let us start by showing why this lemma implies Theorem~4.1. For the same reasons as in~\ref{sec:4.3}, we may assume that $\mathbf{H} = \mathbf{A}\mathbf{U}' = \mathbf{U}'\mathbf{A}$ where $\mathbf{U}'$~is a unipotent $k$-subgroup of codimension~$1$ in~$\mathbf{U}$ which contains $[\mathbf{U}, \mathbf{U}]$ and that $N$~is the group of $k$-points of a maximal connected nilpotent $k$-subgroup $\mathbf{N}$ of~$\mathbf{G}$. Up to conjugating by an element of~$G$, there are only two possibilities for~$\mathbf{N}$: $\mathbf{N} = \mathbf{V}$ or $\mathbf{N} = \mathbf{A}$ (remember that $\mathbf{V}$ and~$\mathbf{U}$ are conjugate over~$k$). The quotient $U/U'$ is homeomorphic to~$k$, hence we may choose a sequence~$u_n$ in~$U$ such that the image of this sequence in~$U/U'$ tends to infinity.

~

\underline{1st case:} $\mathbf{N} = \mathbf{V}$. Suppose by contradiction that $V \backslash G / U'A$ is quasicompact. Extracting if necessary a subsequence, the image of the sequence~$u_n$ in this quotient converges (to a limit that might not be unique!). We may then find sequences $v_n$ in~$V$, $u'_n$ in~$U'$ and $a_n$ in~$A$ such that the sequence $g_n := v_n u_n u'_n a_n$ converges in~$G$. The previous lemma proves that the sequence $u_n u'_n$ is bounded in~$U$. Contradiction.

~

\underline{2nd case:} $\mathbf{N} = \mathbf{A}$. We proceed in the same fashion. Suppose by contradiction that $A \backslash G / U'A$ is quasicompact. Extracting if necessary a subsequence, the image of the sequence~$v_0 u_n$ in this quotient converges. We may then find sequences $b_n$ in~$A$, $u'_n$ in~$U'$ and $a_n$ in~$A$ such that the sequence $h_n := b_n v_0 u_n u'_n a_n$ converges in~$G$. The previous lemma proves that the sequence $u_n u'_n$ is bounded in~$U$. Contradiction.

\subsection{The sequence~$u_n$}
\label{sec:4.5}

It remains to prove Lemma~\ref{sec:4.4}. The Lie algebra~$\mathbffrakg$ has a decomposition defined on~$k$
\[\mathbffrakg = \mathbffrakg_{-2\alpha} \oplus \mathbffrakg_{-\alpha} \oplus \mathbffrakg_{0} \oplus \mathbffrakg_{\alpha} \oplus \mathbffrakg_{2\alpha}.\]
We denote by $\mathbffrakg_k$ the Lie algebra of~$G$. We say that $\Sigma$~is reduced if $\mathbffrakg_{\pm 2\alpha} = 0$. For $\beta$ in~$\Sigma \cap \{0\}$, we call $p_\beta$ the projection onto~$\mathbffrakg_\beta$ parallel to the other subspaces~$\mathbffrakg_{\beta'}$; we call $(\mathbffrakg_\beta)_k$ the intersection of~$\mathbffrakg_\beta$ with~$\mathbffrakg_k$. We call $H_o$ the element of $\mathbffraka := \operatorname{Lie}(\mathbf{A})$ such that $d \alpha(H_o) = 2$. For $p$ in~$\mathbb{Z}$, $\mathbffrakg_{p \alpha}$~is also the eigenspace of~$\ad H_o$ for the eigenvalue~$2p$.

The following lemma is a variant of the Jacobson-Morozov theorem. It holds for any graded semisimple Lie algebra on a field of characteristic zero.

\begin{lemma*}
($\characteristic(k) = 0$) Let $X$ be a nonzero element of $(\mathbffrakg_{q\alpha})_k$ with $q \neq 0$. Then there exists an $\SL(2)$-triple $(Y, H, X)$ with $Y$ in~$(\mathbffrakg_{-q\alpha})_k$ and $H$ in~$(\mathbffrakg_0)_k$.
\end{lemma*}

\begin{proof}
We initially follow the proof of the Jacobson-Morozov theorem (see for example \cite{Bou}~VIII.11.2). Since $X$~is nilpotent, we may find $H$ in~$\ad X(\mathbffrakg_k)$ such that~$[H, X] = 2X$. Replacing if necessary $H$ by~$p_0(H)$, we may assume that $H$ is in~$(\mathbffrakg_0)_k$. We may then complete $(H, X)$ to an $\SL(2)$-triple $(Y, H, X)$ (loc. cit.). Replacing if necessary $Y$ by~$p_{-q\alpha}(Y)$, we may assume that $Y$~is in~$(\mathbffrakg_{-q\alpha})_k$.
\end{proof}

\begin{remark*}
In our situation $\mathbf{G}$~has $k$-rank~$1$. It follows since $H$ generates the Lie algebra of a $k$-split torus that $H$~is in~$\mathbffraka$ and $d\alpha(H) = 2/q$. Hence $H = q^{-1}H_o$. When $q > 0$, the theory of $\SL(2)$-modules (\cite{Bou}~VIII.1.3) proves that the restriction of~$\ad X$ to $\mathbffrakg_{-2\alpha} \oplus \mathbffrakg_{-\alpha}$ is injective and that $\mathbffrakg_{\alpha} \oplus \mathbffrakg_{2\alpha}$ is contained in the image of~$\ad X$; when $q < 0$, the same statement holds with $\alpha$ and~$-\alpha$ exchanged.
\end{remark*}

\begin{proof}[Proof of Lemma~\ref{sec:4.4} when $\Sigma$~is reduced]
We choose $v_0$ of the form~$e^Y$ with $Y$~a nonzero element of~$(\mathbffrakg_{-\alpha})_k$. Note that the exponential is well-defined because $Y$ is nilpotent. The previous discussion proves that the restriction of $\ad Y$ to~$\mathbffrakg_\alpha$ is injective. We write $u_n = e^{X_n}$ with $X_n$ in~$(\mathbffrakg_\alpha)_k$.
\begin{enumerate}[label=\alph*)]
\item We have the equality
\[p_\alpha(\Ad g_n(H_o))
= p_\alpha(\Ad(v_n u_n a_n)(H_o))
= \ad X_n(H_o)
= -2X_n.\]
Hence the sequence~$X_n$ converges and so does~$u_n$.
\item We have the equality
\[
\displayindent0pt
\displaywidth\textwidth
p_0(\Ad h_n(H_o))
= p_0(\Ad(b_n v_0 u_n a_n)(H_o))
= H_o + \ad Y(\ad X_n(H_o))
= H_o - 2\ad Y(X_n).\]
Hence the sequence~$\ad Y(X_n)$ converges, so does~$X_n$ and so does~$u_n$. \qedhere
\end{enumerate}
\end{proof}

\begin{proof}[Proof of Lemma~\ref{sec:4.4} when $\Sigma$~is not reduced]
We choose $v_0$ of the form~$e^Y$ with $Y$ some nonzero element of~$(\mathbffrakg_{-2\alpha})_k$. As previously, the restriction of~$\ad Y$ to~$\mathbffrakg_{2\alpha}$ is injective. We write $u_n = e^{X_n + Z_n}$ with $X_n$ in~$(\mathbffrakg_\alpha)_k$ and $Z_n$ in~$(\mathbffrakg_{2\alpha})_k$. We set $\psi(X_n) := (\ad X_n)^2 \circ p_0$.
\begin{enumerate}[label=\alph*)]
\item We have the equality
\[p_{2\alpha}(\Ad g_n(H_o))
= p_{2\alpha}(\Ad(v_n u_n a_n)(H_o))
= \frac{1}{2}(\ad X_n)^2(H_o) + \ad Z_n(H_o)
= -4Z_n.\]
Hence the sequence~$Z_n$ converges. More generally, we have the equality
\[p_{2\alpha} \circ \Ad g_n \circ p_0 = \frac{1}{2}\psi(X_n) + \ad Z_n \circ p_0.\]
Hence the sequence $\psi(X_n)$ converges. The lemma below proves that the sequence $X_n$ is bounded, hence so is~$u_n$.
\item We have the equality
\[
\displayindent0pt
\displaywidth\textwidth
p_0(\Ad h_n(H_o))
= p_0(\Ad(b_n v_0 u_n a_n)(H_o))
= H_o + \ad Y(\ad Z_n(H_o))
= H_o - 4\ad Y(Z_n).\]
Hence the sequence~$Z_n$ converges. More generally, we have the equality
\[p_0 \circ \Ad h_n \circ p_0 = \frac{1}{2} \ad Y \circ \psi(X_n) + \ad Y \circ \ad Z_n \circ p_0.\]
Hence the sequence $\psi(X_n)$ converges, $X_n$ is bounded, and so is~$u_n$. \qedhere
\end{enumerate}
\end{proof}
We used the following lemma.
\begin{lemma*}
($\characteristic(k) = 0$, $\rank_k(\mathbf{G}) = 1$ and $\Sigma$ not reduced) The map
\[\fundef{\psi:}{(\mathbffrakg_\alpha)_k}{\Hom_k(\mathbffrakg_0, \mathbffrakg_{2\alpha})}{X}{\psi(X) = (\ad X)^2 \circ p_0}\]
is proper.
\end{lemma*}
\begin{proof}
The previous lemma and its remark prove that if $X$ is nonzero in~$(\mathbffrakg_\alpha)_k$, the space $\mathbffrakg_{2\alpha}$ is contained in the image of~$(\ad X)^2$ hence $\psi(X)$~is nonzero. Our lemma is then a consequence of the following elementary claim whose proof we omit.
\end{proof}
\begin{claim*}
Let $E$, $F$ be two finite-dimensional $k$-vector spaces and $\psi: E \to F$ a continuous map that is homogeneous of degree~$2$ (\ie $\psi(\lambda x) = \lambda^2 \psi(x)$ for every $\lambda$ in~$k$ and $x$ in~$E$) and such that $\psi^{-1}(0) = \{0\}$. Then $\psi$ is proper.
\end{claim*}

\section{Properness criterion}
\label{sec:5}

The goal of this section is to prove Theorem~\ref{sec:5.2} which generalizes Proposition~\ref{sec:1.5}. We reuse the notations from~\ref{sec:2.2}.

\subsection{Inclusion modulo compacts}
\label{sec:5.1}

\begin{proposition*}
Let $\mathbf{G}$ be a simply-connected semisimple $k$-group, $G := \mathbf{G}_k$ and $\mu: G \to A^+$ a Cartan projection. Then for every compact subset~$L$ of~$G$, there exists a compact subset~$M$ of~$A$ such that for every $g$ in~$G$, we have $\mu(L g L) \subset \mu(g) M$.
\end{proposition*}

\begin{proof}
Let us fix a compact subset~$L$ of~$G$ such that $L = L^{-1}$. Let $(V, \rho)$ be an irreducible representation of~$G$ with highest $k$-weight~$\chi$. By Lemma~\ref{sec:2.3}, it suffices to show that there exists a constant $C > 1$ such that for every $g$ in~$G$ and $l_1$, $l_2$ in~$L$, we have
\[C^{-1} \leq |\chi(\mu(l_1 g l_2))| \cdot |\chi(\mu(g))|^{-1} \leq C.\]

According to Lemma~\ref{sec:2.4}, it suffices to find a constant $C' > 1$ such that for every $g$ in~$G$ and $l_1$, $l_2$ in~$L$, we have
\[C'^{-1} \leq \|\rho(l_1 g l_2)\| \cdot \|\rho(g)\|^{-1} \leq C'.\]

We may take $C' = \sup_{l \in L} \|\rho(l)\|^2$.
\end{proof}

\subsection{Properness criterion}
\label{sec:5.2}

\begin{theorem*}
Let $k$~be a local field, $\mathbf{G}$ a simply-connected semisimple $k$-group, $G := \mathbf{G}_k$ and $\mu: G \to A^+$ a Cartan projection.

Let $H_1$, $H_2$ be two closed subsets of~$G$. The pair~$(H_1, H_2)$ is $G$-proper if and only if the pair~$(\mu(H_1), \mu(H_2))$ is $A$-proper.
\end{theorem*}

The following corollary is a reformultation of this theorem for subgroups, made possible by Lemma~\ref{lemma_3.1.1}.

\begin{corollary*}
We keep the notations of the theorem. Let $H_1$, $H_2$ be two closed subgroups of~$G$. The group~$H_1$ acts properly on~$G/H_2$ if and only if for every compact subset~$M$ of~$A$, the set $\mu(H_1) \cap \mu(H_2)M$ is compact.
\end{corollary*}

\begin{remark*}
Let us give a geometric interpretation of this criterion when $k = \mathbb{R}$. Endow~$A$ with an $A$-invariant Riemannian metric and denote by~$d$ the corresponding distance. The criterion is that $(H_1, H_2)$ is $G$-proper if and only if for every $R > 0$, the set of pairs of points $(a_1, a_2)$ in $\mu(H_1) \times \mu(H_2)$ such that $d(a_1, a_2) \leq R$ is compact. In other terms, \emph{$\mu(H_1)$ and $\mu(H_2)$ get infinitely far apart from each other when approaching infinity.}
\end{remark*}

\begin{proof}
Suppose first that $(H_1, H_2)$ is $G$-proper. By definition, for~$j = 1, 2$, $\mu(H_j)$~is contained in~$H_j$ modulo the compacts of~$G$ (see~\ref{sec:3.1}); we also have the opposite inclusion... but we will not need it. Lemma~\ref{lemma_3.1.2} proves that the pair~$(\mu(H_1), \mu(H_2))$ is $G$-proper, hence it is $A$-proper.

Conversely, suppose that $(\mu(H_1), \mu(H_2))$ is $A$-proper. Let $L$~be a compact subset of~$G$ such that $L = K L K$ and $M$~be a compact subset of~$A$ as given by Proposition~\ref{sec:5.1}. We have the inclusion
\[\mu(H_1 \cap L H_2 L) \subset \mu(H_1) \cap \mu(H_2) M.\]
Since $\mu$~is proper, $H_1 \cap L H_2 L$ is compact and $(H_1, H_2)$~is $G$-proper.
\end{proof}

\subsection{Examples}
\label{sec:5.3}

Here are a few particular cases of this theorem.

The first one is due to Kobayashi (\cite{Kob1}~4.1) when $k = \mathbb{R}$.

\begin{corollary}
\label{corollary_5.3.1}
With the notations of~\ref{sec:2.2}. Let $\mathbf{G}$ be a semisimple $k$-group, $\mathbf{A}$ a maximal $k$-split torus of~$\mathbf{G}$, $\mathbf{H}_1$ and~$\mathbf{H_2}$ two reductive $k$-subgroups of~$\mathbf{G}$. For $j = 1, 2$, we denote by~$\mathbf{A}_j$ a maximal $k$-split torus of~$\mathbf{H}_j$. We suppose that the $\mathbf{A}_j$ are contained in~$\mathbf{A}$ (we easily reduce the problem to this case, by conjugating the~$\mathbf{H}_j$ by some element of~$G$).

We then have the equivalence:
\[H_1 \text{ acts properly on } G/H_2 \iff \forall w \in W,\; A_1 \cap w A_2 \text{ is finite.}\]
\end{corollary}

\begin{proof}
We may assume $\mathbf{G}$ to be simply connected. By the Cartan decomposition for~$H_j$, the group~$H_j$ is contained in~$A_j$ modulo the compacts of~$G$. We consequently have the equivalences:
\begin{align*}
&H_1 \text{ acts properly on } G/H_2 & \\
\iff &A_1 \text{ acts properly on } G/A_2 &\text{(by Lemma~\ref{lemma_3.1.2})} \\
\iff &(\mu(A_1), \mu(A_2)) \text{ is $A$-proper} &\text{(by the theorem)} \\
\iff &\forall w \in W,\; A^0 \cap (A_1 \cap w A_2) \text{ is compact} &\text{(since $\mu(A_j) = \bigcup_{w \in W}(w A_j \cap A^+)$)} \\
\iff &\forall w \in W,\; A_1 \cap w A_2 \text{ is finite.} &\qedhere
\end{align*}
\end{proof}

\begin{remark*}
When $k = \mathbb{R}$, the corollary and its proof still hold under the weaker assumption that $H_i$~is reductive in~$G$ (\ie $H_i$~is a connected closed subgroup of~$G$ such that the adjoint action of~$H_i$ on the Lie algebra of~$G$ is semisimple). This latter formulation is the one used in~\cite{Kob1}.
\end{remark*}

The second corollary is due to Friedland~(\cite{Fr}) when $k = \mathbb{R}$.

\begin{corollary}
\label{corollary_5.3.2}
Let $G = \SL(n, k)$, $H = \SL(m, k) \times I_{n-m}$. A closed subgroup~$\Gamma$ of~$G$ acts properly on~$G/H$ if and only if for every compact subset $C$ of~$k$, the set
\[\Gamma^H_C := \setsuch{g \in \Gamma}{\text{$g$ has $m$ singular values in $C$}}\]
is compact.
\end{corollary}

\begin{proof}
Let us take the notations of Example~\ref{sec:2.2} and let us write, for $a \in A^+$,
\[a = \begin{pmatrix}
\sigma_1 & & 0 \\
& \ddots & \\
0 & & \sigma_n
\end{pmatrix}.\]
This corollary is a consequence of Corollary~5.2 because the Weyl group is the group of permutations of the coordinates and because we have
\[\mu(H) = \setsuch{a \in A^+}{\exists i \in \{1, \ldots, n-m+1\},\; \sigma_i = \sigma_{i+1} = \cdots = \sigma_{i+m-1} = 1}.\]
\end{proof}

Let us give just three more examples... but we could easily continue the list.

\begin{corollary}
\label{corollary_5.3.3}
The statement of Corollary~\ref{corollary_5.3.2} also holds for $G = \SL(n, k)$ and
\begin{enumerate}[label=\alph*)]
\item $H = \SL(m, k) \times \SL(n-m, k)$, if we take
\[\Gamma^H_C := \setsuch{g \in \Gamma}{\text{there exist $m$ singular values of~$g$ whose product is in~$C$}};\]
\item $n = 2m$ and $H = \Sp(m, k)$, if we take
\[\Gamma^H_C := \setsuch{g \in \Gamma}{\forall i = 1, \ldots, m,\; \sigma_i(g)\sigma_{2m+1-i}(g) \in C};\]
\item $\characteristic(k) \neq 2$ and $H = \SO(b)$ is the stabilizer of a nondegenerate symmetric bilinear form~$b$ of index~$d$, if we take
\[\Gamma^H_C := \setsuch{g \in \Gamma}{\begin{cases}
\forall i = 1, \ldots, d,\; &\sigma_i(g)\sigma_{n+1-i}(g) \in C; \\
\forall j = d+1, \ldots, n-d,\; &\sigma_j(g) \in C.
\end{cases}}.\]
\end{enumerate}
\end{corollary}

\begin{proof}
This is a consequence of Corollary~\ref{corollary_5.3.2} since we have, respectively in each of the three cases:
\begin{align*}
\mu(H) &= \setsuch{a \in A^+}{\exists i_1 < \cdots < i_m,\; \sigma_{i_1} \cdots \sigma_{i_m} = 1}; \\
\mu(H) &= \setsuch{a \in A^+}{\forall i = 1, \ldots, m,\; \sigma_i \sigma_{2m+1-i} = 1}; \\
\mu(H) &= \setsuch{a \in A^+}{\begin{cases}
\forall i = 1, \ldots, d,\; &\sigma_i \sigma_{n+1-i} = 1; \\
\forall j = d+1, \ldots, n-d,\; &\sigma_j =1.
\end{cases}}. \qedhere
\end{align*}
\end{proof}

\section{Proximality}
\label{sec:6}

This chapter mostly consists of preliminaries about proximal maps which will play a central role in defining and studying the properties of ``$\eps$-Schottky'' groups in the next chapter.

\subsection{Notations}
\label{sec:6.1}
Let $k$~be a local field, $V$ a finite-dimensional $k$-vector space, $X := \mathbb{P}(V)$ the projective space of~$V$: it is the set of vector lines in~$V$. We endow~$V$ with a norm $\|.\|$, and we define on~$X$ a distance~$d$ by
\[d(x_1, x_2) := \inf \setsuch{\|v_1 - v_2\|}{v_i \in x_i \text{ and } \|v_i\| = 1,\; \forall i = 1, 2}.\]
If $X_1$ and~$X_2$ are two closed subsets of~$X$, we set
\[\delta(X_1, X_2) = \inf \setsuch{d(x_1, x_2)}{x_1 \in X_1, x_2 \in X_2},\]
and we denote by
\[d(X_1, X_2) = \sup \setsuch{\delta(x_i, X_{3-i})}{x_i \in X_i \text{ and } i = 1, 2}\]
the Hausdorff distance between $X_1$ and~$X_2$.

The following lemma will be useful to us.

\begin{lemma*}
For every $\eps > 0$, there exists a constant $r_\eps > 0$ such that, for every hyperplane~$V'$ of~$V$ and for every pair of points $v_1$, $v_2$ of~$V$ of norm~$1$ satisfying $\delta(k v_i, \mathbb{P}(V')) \geq \eps$ for $i = 1, 2$, the number $\alpha \in k$ defined by $v_1 - \alpha v_2 \in V'$ satisfies $r_\eps^{-1} \leq |\alpha| \leq r_\eps$.
\end{lemma*}

\begin{proof}
This follows from compactness of the set of such triples $(V', v_1, v_2)$ and from continuity of the map $(V', v_1, v_2) \mapsto |\alpha|$.
\end{proof}

\subsection{$\eps$-proximality}
\label{sec:6.2}

For $g$ in $\End(V) \setminus \{0\}$, we write $\lambda_1(g) := \sup \setsuch{|\alpha|}{\alpha \text{ eigenvalue of } g}$. Of course an eigenvalue of~$g$ is generally in a finite extension~$k'$ of~$k$. We have implicitly introduced on this extension the unique absolute value that extends the absolute value of~$k$. Note that $\lambda_1(g) \leq \|g\|$.

\begin{definition*}
The element~$g$ is said to be proximal if it has a unique eigenvalue~$\alpha$ such that $|\alpha| = \lambda_1(g)$ and if this eigenvalue has multiplicity~$1$. This eigenvalue~$\alpha$ is then in~$k$ and we call $x^+_g \in X$ the corresponding eigenline. We call $v^+_g$ a vector of~$x^+_g$ of norm~$1$, $V^<_g$~the $g$-invariant hyperplane transverse to~$x^+_g$ and $X^<_g := \mathbb{P}(V^<_g)$.
\end{definition*}

The set of proximal maps is an open subset of $\End(V)$, for the norm topology. On this open set, the maps $g \mapsto x^+_g$ and $g \mapsto X^<_g$ are continuous.

In the following definition, we impose a uniform control on proximality. This definition is very close to that of \cite{A-M-S}. We fix $\eps > 0$ and we call
\[b^\eps_g := \setsuch{x \in X}{d(x, x^+_g) \leq \eps};\]
\[B^\eps_g := \setsuch{x \in X}{\delta(x, X^<_g) \geq \eps}.\]
Note that $b^\eps_g$ is contained in~$B^\eps_g$ as soon as $\delta(x^+_g, X^<_g) \geq 2\eps$.

\begin{definition*}
A proximal element $g$ is said to be $\eps$-proximal if $\delta(x^+_g, X^<_g) \geq 2\eps$, $g(B^\eps_g) \subset b^\eps_g$ and $\restr{g}{B^\eps_g}$ is $\eps$-Lipschitz.
\end{definition*}

\begin{remarks*}~
\begin{itemize}
\item The ``$\eps$-Lipschitz'' condition means that for every $x$, $y$ in~$B^\eps_g$, $d(g x, g y) \leq \eps d(x, y)$.
\item If $g$ is $\eps$-proximal then $g^n$ is $\eps$-proximal, for every $n \geq 1$.
\item For every proximal element~$g$ and for every $\eps > 0$ such that $2\eps \leq \delta(x^+_g, X^<_g)$, there exists $n_0 \geq 1$ such that for every $n \geq n_0$, $g^n$ is $\eps$-proximal.
\item However, there exist proximal elements which are not $\eps$-proximal for any value of~$\eps$.
\end{itemize}
\end{remarks*}

Here is a sufficient condition for proximality.

\begin{lemma*}
Let $g$~be in $\End(V) \setminus \{0\}$, $x^+$ in~$\mathbb{P}(V)$, $W$~a hyperplane of~$V$ and $\eps > 0$.  We write $Y := \mathbb{P}(W)$, $b^\eps := \setsuch{x \in X}{d(x, x^+) \leq \eps}$ and $B^\eps := \setsuch{x \in X}{\delta(x, Y) \geq \eps}$. Suppose that $\delta(x^+, Y) \geq 6\eps$, $g(B^\eps) \subset b^\eps$ and $\restr{g}{B^\eps}$ is $\eps$-Lipschitz.

Then $g$ is $2\eps$-proximal, $d(x^+_g, x^+) \leq \eps$ and $d(X^<_g, Y) \leq \eps$.
\end{lemma*}

\begin{proof}
The restriction of~$g$ to~$B^\eps$ is an $\eps$-Lipschitz contraction. It consequently has an attracting fixed point~$x^+_g$. Hence $g$~is proximal. Since $g(B^\eps) \subset b^\eps$, we have $d(x^+_g, x^+) \leq \eps$. Since $B^\eps$ is contained in the basin of attraction of~$x^+_g$, we have $X^<_g \cap B^\eps = \emptyset$, or in other terms $d(X^<_g, Y) \geq \eps$.

We deduce that $\delta(x^+_g, X^<_g) \geq 4\eps$, then that $g(B^{2\eps}_g) \subset g(B^\eps) \subset b^\eps \subset b^{2\eps}_g$ and finally that $\restr{g}{B^{2\eps}_g}$ is $\eps$-Lipschitz. Hence $g$ is $2\eps$-proximal.
\end{proof}

\subsection{Norm and largest eigenvalue}
\label{sec:6.3}

\begin{lemma*}
Let $\eps > 0$. The set of $\eps$-proximal elements of~$\End(V)$ is a closed subset of $\End(V) \setminus \{0\}$.
\end{lemma*}

\begin{proof}
Let $(g_n)$ be a sequence of $\eps$-proximal maps that converges to some nonzero element~$g$. Let us first show that $g$~is proximal. Extracting if necessary a subsequence, we lose no generality in assuming that $\lim_{n \to \infty} x^+_{g_n} = x^+$ and $\lim_{n \to \infty} X^<_{g_n} = Y$, with the latter limit taken with respect to Hausdorff distance. We introduce the notations $b^\eps := \setsuch{x \in X}{d(x, x^+) \leq \eps}$ and $B^\eps := \setsuch{x \in X}{\delta(x, Y) \geq \eps}$. We then have $\lim_{n \to \infty} b^\eps_{g_n} = b^\eps$ and $\lim_{n \to \infty} B^\eps_{g_n} = B^\eps$. It follows that $\delta(x^+, Y) \geq 2\eps$, $g(B^\eps) \subset b^\eps$ and $\restr{g}{B^\eps}$ is an $\eps$-Lipschitz contraction. Hence $g$ has an attracting fixed point~$x^+_g$ in~$b^\eps$. This shows that $g$~is proximal.

The continuity of the maps $g \mapsto x^+_g$ and $g \mapsto X^<_g$ ensures that $x = x^+_g$ and $Y = X^<_g$, hence that $g$~is $\eps$-proximal.
\end{proof}

The following corollary says that for an $\eps$-proximal element~$g$, $\lambda_1(g)$ is a good approximation for the norm of~$g$.

\begin{corollary*}
Let $V$ be a finite-dimensional $k$-vector space and $\eps > 0$. There exists a constant $c_\eps \in (0, 1)$ such that for every $\eps$-proximal linear transformation~$g$ of~$V$, we have
\[c_\eps \|g\| \leq \lambda_1(g) \leq \|g\|.\]
\end{corollary*}

\begin{proof}
This follows from compactness of the set of $\eps$-proximal linear transformations of~$V$ having norm~$1$ as well as from continuity of the map $g \mapsto \lambda_1(g)$.
\end{proof}

\subsection{Product of $\eps$-proximal elements}
\label{sec:6.4}

The following proposition gives an approximation of~$\|g\|$ when $g$~is a word whose letters are $\eps$-proximal elements.

\begin{proposition*}
For every $\eps > 0$, there exists a constant $C_\eps > 0$ with the following property. Let $g_1, \ldots, g_l$ be any tuple of $\eps$-proximal linear transformations of~$V$ satisfying (using the convention $g_0 := g_l$):
\[\delta(x^+_{g_{j-1}}, X^<_{g_j}) \geq 6\eps \quad \text{for } j = 1, \ldots, l.\]
Then for any $n_1, \ldots, n_l \geq 1$, the product $g = g_l^{n_l} \cdots g_1^{n_1}$ is $2\eps$-proximal and we have
\[C_\eps^{-l} \leq \lambda_1(g) \cdot \prod_{1 \leq j \leq l} \lambda_1(g_j)^{-n_j} \leq C_\eps^l\]
and
\[C_\eps^{-l-1} \leq \|g\| \cdot \prod_{1 \leq j \leq l} \lambda_1(g_j)^{-n_j} \leq C_\eps^{l+1}.\]
\end{proposition*}

\begin{proof}
Let $x^+_j$, $v^+_j$, $X^<_j$, $V^<_j$, $B^\eps_j$, $b^\eps_j$ denote respectively $x^+_{g_j}$, $v^+_{g_j}$ etc. Given that $g_j^n$ is $\eps$-proximal, that $x^+_{g_j^n} = x^+_j$ and that $X^<_{g_j^n} = X^<_j$, we may assume that $n_j = 1$ for every $j = 1, \ldots, l$.

We have the inclusion $g_1(B^\eps_1) \subset b^\eps_1 \subset B^\eps_2$ since $\delta(x^+_1, X^<_2) \geq 2\eps$. Similarly we have $g_2 g_1(B^\eps_1) \subset b^\eps_2$; by iterating, we get that $g(B^\eps_1) \subset b^\eps_l$ and that $\restr{g}{B^\eps_1}$ is $\eps$-Lipschitz. We may apply Lemma~\ref{sec:6.2} since $\delta(x^+_l, X^<_1) \geq 6\eps$. We get that $g$~is $2\eps$-proximal and that $x^+_g \in b^\eps_l$.

Let $w_0 := v^+_g$, $y_0 := x^+_g$ and, for $j = 1, \ldots, l$,
\[w_j := g_j w_{j-1} \text{ and } y_j := g_j y_{j-1}.\]
By construction, we have
\[\begin{cases}
y_j \in b^\eps_j \quad\text{for } j = 0, \ldots, l; \\
w_l = \lambda_1(g)w_0.
\end{cases}\]
For $j = 1, \ldots, l$, let $\alpha_j \in k$ be the number defined by the equality
\[w_{j-1} = \alpha_j v^+_j \text{ modulo } V^<_j.\]
Since $\delta(y_{j-1}, X^<_j) > 5\eps$ and $\delta(x^+_j, X^<_j) \geq 2\eps$, Lemma~\ref{sec:6.1} shows that
\[r_\eps^{-1} \leq \frac{|\alpha_j|}{\|w_{j-1}\|} \leq r_\eps.\]
We also have
\[w_j = \alpha_j \lambda_1(g_j) v^+_j \text{ modulo } V^<_j.\]
Since $\delta(y_j, X^<_j) \geq \eps$, the same Lemma~\ref{sec:6.1} shows that
\[r_\eps^{-1} \leq \frac{|\alpha_j|\lambda_1(g_j)}{\|w_j\|} \leq r_\eps.\]
These two inequalities yield
\[r_\eps^{-2} \leq \frac{\|w_j\|}{\|w_{j-1}\|}\lambda_1(g_j)^{-1} \leq r_\eps^2.\]
Multiplying these $l$~inequalities together and remarking that $\frac{\|w_l\|}{\|w_0\|} = \lambda_1(g)$, we get
\[r_\eps^{-2l} \leq \lambda_1(g) \cdot \prod_{1 \leq j \leq l} \lambda_1(g_j)^{-1} \leq r_\eps^{2l},\]
and then using Corollary~\ref{sec:6.3}
\[r_\eps^{-2l} \leq \|g\| \cdot \prod_{1 \leq j \leq l} \lambda_1(g_j)^{-1} \leq r_\eps^{2l}c_\eps^{-1}.\]
This proves our proposition if we set $C_\eps := \max(r_\eps^2, c_\eps^{-1})$.
\end{proof}

\subsection{$\eps$-Schottky subgroup on~$\mathbb{P}(V)$}
\label{sec:6.5}
The following definition is motivated by Proposition~\ref{sec:6.4}.

\begin{definition*}
Let $\eps > 0$ and $t \geq 2$. We say that a subsemigroup (resp. subgroup) $\Gamma$ of~$\GL(V)$ with generators~$\gamma_1, \ldots, \gamma_t$ is $\eps$-Schottky on~$\mathbb{P}(V)$ if it satisfies the following properties. We set $E := \{\gamma_1, \ldots, \gamma_t\}$ (resp. $E := \{\gamma_1, \ldots, \gamma_t, \gamma_1^{-1}, \ldots, \gamma_t^{-1}\}$).
\begin{enumerate}[label=\roman*)]
\item For any $h$ in~$E$, $h$~is $\eps$-proximal.
\item For any $h$, $h'$ in~$E$ (resp. $h$, $h'$ in~$E$ with $h' \neq h^{-1}$), $\delta(x^+_h, X^<_{h'}) \geq 6\eps$.
\end{enumerate}
\end{definition*}

\begin{remarks*}~
\begin{itemize}
\item Of course, this definition depends on the choice of the generators~$\gamma_j$ and of the norm on~$V$.
\item If the semigroup (resp. group) $\Gamma$ with generators $\gamma_1, \ldots, \gamma_t$ is $\eps$-Schottky on~$\mathbb{P}(V)$, then so is the semigroup (resp. group) $\Gamma_m$ with generators $\gamma_1^m, \ldots, \gamma_t^m$, for every $m \geq 1$.
\item A subgroup~$\Gamma$ with generators $\gamma_1, \ldots, \gamma_t$ that is $\eps$-Schottky on $\mathbb{P}(V)$ is discrete in~$\GL(V)$ and is a free group on these generators $\gamma_1, \ldots, \gamma_t$. This follows from the ping-pong lemma (see \cite{Ti2}~1.1).
\end{itemize}
\end{remarks*}

\section{Schottky groups}
\label{sec:7}

The goal of this section is to prove Theorem~\ref{sec:7.4} which generalizes Proposition~\ref{sec:1.5}. We then deduce Theorem~\ref{sec:7.5} which generalizes Theorem~\ref{sec:1.1}. We finish this section by proving the corollaries from the introduction.

\subsection{Jordan decomposition}
\label{sec:7.1}
Let $\mathbf{G}$ be a semisimple $k$-group with $k$-rank $r \geq 1$ (in other terms $\mathbf{G}$~is isotropic) and $G := \mathbf{G}_k$. In order to simplify some formulations, we shall embed the semigroup~$A^+$ into a salient convex cone with nonempty interior~$A^*$, contained in some $r$-dimensional $\mathbb{R}$-vector space~$A^\bullet$.

When $k = \mathbb{R}$ or~$\mathbb{C}$, we set $A^\bullet := A^o$ and $A^* := A^+$ (see \ref{sec:2.2}). The identification of~$A^\bullet$ with its Lie algebra makes it an $\mathbb{R}$-vector space.

When $k$~is non-Archimedean, $A^o$ is a free $\mathbb{Z}$-module of rank~$r$. We set $A^\bullet := A^o \otimes_\mathbb{Z} \mathbb{R}$ and we define~$A^*$ to be the convex hull of~$A^+$ in~$A^\bullet$.

The goal of this subsection is to define a map $\lambda: G \to A^*$ that we shall call the \emph{Lyapunov projection}. Though it is not strictly necessary, it makes things clearer to treat the Archimedean and non-Archimedean cases separately.

Suppose first that $k = \mathbb{R}$ or~$\mathbb{C}$. This is the easier case. Every element~$g$ of~$G$ has a unique decomposition~$g = g_e g_h g_u$ into a product of three pairwise commuting elements of~$G$, with $g_e$~elliptic, $g_h$~hyperbolic and $g_u$~unipotent (see for example \cite{Kos}~2.1). We set $\lambda(g)$ to be the unique element of~$A^+$ that is conjugate to~$g_h$.

Suppose now that $k$~is non-Archimedean. Then such a decomposition of~$g$ does not always exist. Example: $g = \begin{pmatrix} h & 0 \\ 0 & h^{-1} \end{pmatrix} \in G = \SL(4, \mathbb{Q}_p)$ where $h = \begin{pmatrix} 1 & p-1 \\ 1 & -1 \end{pmatrix}$ has characteristic polynomial $X^2-p$. However, we have the following lemma which is probably well-known.

\begin{lemma*}
There exists $n \geq 1$ such that for every element~$g$ of~$G$, the element $g' := g^n$ has a unique decomposition $g' = g'_e g'_h g'_u$ into a product of three pairwise commuting elements of~$G$, with $g'_e$~elliptic (\ie $\Ad(g'_e)$~is semisimple with eigenvalues of modulus~$1$), $g'_u$~unipotent and $g'_h$~conjugate to an element~$a'$ of~$A^+$. This element $a'$ is unique.
\end{lemma*}

\begin{definition*}
We set $\lambda(g) := \frac{1}{n}a' \in A^*$.
\end{definition*}

It is clear that $\lambda(g)$~does not depend on the choice of~$n$.

\begin{proof}
Let us realize~$\mathbf{G}$ as a $k$-subgroup of~$\mathbf{SL}(V)$ where $V$~is a $k$-vector space of dimension~$d$ whose $k$-weights generate $X^*(\mathbf{A})$.

Let $g$~be an element of~$G$. It has a unique so-called Jordan decomposition $g = g_s g_u$ into a product of two commuting elements of~$\mathbf{G}$, with $g_s$~semisimple and $g_u$~unipotent. We may assume that $g_s$ and~$g_u$ are in~$G$: this always holds when $\characteristic(k) = 0$; if  $\characteristic(k) = p > 0$, it suffices to replace~$g$ by~$g^{p^d}$ since $g^{p^d}_u = 1$.

Let $n := d!$. The moduli of the eigenvalues of $g' := g^n$ are in~$|\pi|^\mathbb{Z}$. Let $g' = g'_s g'_u$ be the Jordan decomposition of~$g'$. Then there exists a unique decomposition of~$g'_s$ into a product $g'_s = g'_e g'_h$ of two commuting semisimple elements of~$\SL(V)$ such that $g'_e$ (resp.~$g'_h$) has eigenvalues of modulus~$1$ (resp. in~$k^o$). By construction, every vector subspace of the algebra $k[V]$ of polynomial functions on~$V$ that is $g'_s$-invariant is also $g'_e$-~and $g'_h$-invariant. It follows that $g'_e$ and~$g'_h$ are in~$G$ and commute with~$g'_u$. The element $g'_h$ is in a $k$-split one-dimensional torus. Hence it is conjugate to an element~$a'$ of~$A$. Since the eigenvalues of~$a'$ are in~$k^o$, $a'$~is in~$A^o$. Replacing if necessary $a'$ by some~$w \cdot a'$ where $w$~is in the Weyl group, the element~$a'$ is in~$A^+$.

Uniqueness of~$a'$ is clear since, with the notations of~\ref{sec:2.3}, for every $i = 1, \ldots, r$, $|\omega_i(a')|$~is the largest among the moduli of the eigenvalues of~$\rho_i(g')$.
\end{proof}

The opposition involution $\iota: A^+ \to A^+$ uniquely extends to an $\mathbb{R}$-linear map from~$A^\bullet$ to itself that preserves~$A^*$, which we shall also denote by~$\iota$. We have the equality, for every~$g$ in~$G$:
\[\lambda(g^{-1}) = \iota(\lambda(g)).\]

\begin{definition*}
We say that a subset~$\Omega$ of~$A^+$ is a convex open cone if $\Omega$~is the intersection of~$A^+$ with some convex open cone~$\Omega^\bullet$ in the $\mathbb{R}$-vector space~$A^\bullet$.

For $g$ in~$G$ such that $\lambda(g) \neq 1$, we call $\Lambda_g$ the half-line of~$A^\bullet$ containing~$\lambda(g)$.
\end{definition*}

\subsection{$\eps$-Schottky groups}
\label{sec:7.2}

We choose $r$ irreducible representations $(V_i, \rho_i)$ of~$G$ whose highest $k$-weights~$\omega_i$ have multiplicity one and are $r$ independent characters of~$\mathbf{A}$ (Lemma~\ref{sec:2.3}). We fix some norm on each of the $k$-vector spaces~$V_i$.

\begin{definition*}
Let $\eps > 0$ and $t \geq 2$. We say that a subsemigroup (resp. subgroup) $\Gamma$ of~$G$ with generators~$\gamma_1, \ldots, \gamma_t$ is $\eps$-Schottky if, for $i = 1, \ldots, r$, the subsemigroup (resp. subgroup) $\rho_i(\Gamma)$ of~$\GL(V_i)$ with generators $\rho_i(\gamma_1), \ldots, \rho_i(\gamma_t)$ is $\eps$-Schottky on $\mathbb{P}(V_i)$.
\end{definition*}

We then set
\[E_\Gamma := \{\gamma_1, \ldots, \gamma_t\}
\quad \text{(resp. }
E_\Gamma := \{\gamma_1, \ldots, \gamma_t, \gamma_1^{-1}, \ldots, \gamma_t^{-1}\}
\text{).}\]
\begin{remarks*}~
\begin{itemize}
\item Of course this definition depends on the choice of the representations~$\rho_i$, of the norms on~$V_i$ and of the generators~$\gamma_j$.
\item The remarks made in~\ref{sec:6.5} for $\eps$-Schottky groups on~$\mathbb{P}(V)$ are also valid for $\eps$-Schottky groups.
\end{itemize}
\end{remarks*}

\begin{definition*}
A word $w = g_l \cdots g_1$ with $g_j$ in~$E_\Gamma$ is said to be reduced if $g_{j-1} \neq g_j^{-1}$ for $j = 2, \ldots, l$, and very reduced if additionally we have $g_1 \neq g_l^{-1}$.
\end{definition*}

When $\Gamma$~is a Schottky subsemigroup, every word is very reduced since, for any $h, h'$ in~$E_\Gamma$, we have $h' \neq h^{-1}$.

The following lemma allows us to construct $\eps$-Schottky semigroups (resp. groups).

\begin{lemma*} ($\mathbf{G}$~is an isotropic semisimple $k$-group) Let $a_1, \ldots, a_j, \ldots$ be elements of~$A^{++}$.
\begin{enumerate}[label=\alph*)]
\item We may choose elements $\gamma_1, \ldots, \gamma_t$ of~$G$ such that, setting $E := \{\gamma_1, \ldots, \gamma_t\}$ (resp. $E := \{\gamma_1, \ldots, \gamma_t, \gamma_1^{-1}, \ldots, \gamma_t^{-1}\}$), we have:
\begin{enumerate}[label=\roman*)]
\item $\lambda(\gamma_j) = a_j$ for every $j = 1, \ldots, t$. In particular, for every $h$ in~$E$ and $i = 1, \ldots, r$, $\rho_i(h)$~is proximal.
\item $x^+_{\rho_i(h)} \not\in X^<_{\rho_i(h')}$ for any $h$, $h'$ in~$E$ (resp. for any $h$, $h'$ in~$E$ with $h' \neq h^{-1}$) and $i = 1, \ldots, r$.
\item The semigroup generated by~$\gamma_j$ is Zariski-connected, for every $j = 1, \ldots, t$.
\item The semigroup~$\Gamma$ generated by~$E$ is Zariski-dense in~$G$.
\end{enumerate}
\item For every such choice, there exists $m_o \geq 1$ and $\eps > 0$ such that for every $m \geq m_o$, the subsemigroup (resp. subgroup) $\Gamma_m$ with generators $\gamma_1^m, \ldots, \gamma_t^m$ is $\eps$-Schottky and Zariski-dense.
\end{enumerate}
\end{lemma*}

\begin{proof}~
\begin{enumerate}[label=\alph*)]
\item For every element~$a$ of~$A^{++}$ and $i = 1, \ldots, r$, the elements $\rho_i(a)$ and~$\rho_i(a^{-1})$ are proximal. The points $x^+_i := x^+_{\rho_i(a)}$ and $x^-_i := x^+_{\rho_i(a^{-1})}$, as well as the sets $X^<_i := X^<_{\rho_i(a)}$ and $X^>_i := X^<_{\rho_i(a^{-1})}$, do not depend on the choice of~$a$ in~$A^{++}$. Also the semigroup generated by~$a$ is Zariski-connected.

We may assume that $\mathbf{G}$~is simply connected. We then decompose~$\mathbf{G}$ into a product of a $k$-group $\mathbf{G}_{an}$ that is anisotropic and of a $k$-group $\mathbf{G}_{is}$ that has no anisotropic factor. The group~$A$ is contained in~$G_{is}$ and we have $\rho_i(G_{an}) = 1$ for every $i = 1, \ldots, r$.

We shall consruct by induction on~$j$ elements~$\gamma_j$ satisfying i), ii) and~iii). Let $b$ be an element of~$G_{an}$ that generates a Zariski-connected subsemigroup and that is not contained in any proper normal $k$-subgroup of~$\mathbf{G}_{an}$. It suffices to take $\gamma_j = h_j b a_j h_j^{-1}$ where $h_j$ is in the Zariski-open set
\begin{multline*}
U_j := \left\{ h \in G \;\middle|\; \begin{cases}
\rho_i(h) \cdot x^\alpha_i \not\in \rho_i(h_{j'})(X^<_i \cup X^>_i) \\ 
\rho_i(h_{j'}) \cdot x^\alpha_i \not\in \rho_i(h)(X^<_i \cup X^>_i)
\end{cases} \forall \alpha = \pm,\; \forall i = 1, \ldots, r \right. \\
\left. \text{ and }\;
\forall j' = 1, \ldots, j-1
\vphantom{\begin{cases} X^<_i \\ X^<_i \end{cases}} \right\}.
\end{multline*}
This Zariski-open set is nonempty since $G$~is Zariski-connected and $\rho_i$~is an irreducible representation.

It remains to check iv). Let $G_j$~denote the Zariski-closure in~$G$ of the semigroup generated by $\gamma_1, \ldots, \gamma_j$. If $G_{j-1} \neq G$, we can choose~$h_j$ in the Zariski-open set $U_j \cap V_j$ where $V_j := \setsuch{h \in G}{h b a h^{-1} \not\in G_{j-1}}$. This Zariski-open set is nonempty since $b a$ is not contained in any proper normal $k$-subgroup of~$\mathbf{G}$. The increasing sequence of Zariski-connected and Zariski-closed subgroups~$G_j$ is necessarily stationary. Hence we have $G_j = G$ for sufficiently large~$j$.
\item It suffices to take $\eps$ such that, for every $h, h'$ in~$E$ (resp. $h, h'$ in~$E$ with $h' \neq h^{-1}$) and $i = 1, \ldots, r$, we have
\[6\eps \leq \delta(x^+_{\rho_i(h)}, X^<_{\rho_i(h')}).\]
We then use Remark~\ref{sec:6.2} to deduce that the group~$\Gamma_m$ is $\eps$-Schottky with generators $\gamma_1^m, \ldots, \gamma_t^m$. The Zariski-closure of~$\Gamma_m$ contains each~$\gamma_j$ by iii), hence it is equal to~$G$ by iv). \qedhere
\end{enumerate}
\end{proof}

\subsection{Cartan projection of an $\eps$-Schottky group}
\label{sec:7.3}
Now that we know how to construct $\eps$-Schottky groups, we need to calculate their Cartan projection. So let us assume $\mathbf{G}$ is simply connected ant let $\mu: G \to A^+$ be some Cartan projection.
\begin{proposition*}
For every $\eps > 0$, there exists a compact subset $M_\eps$ of~$A^\bullet$ that has the following property.

For every $\eps$-Schottky subsemigroup (resp. subgroup) $\Gamma$ of~$G$ with generators $\gamma_1, \ldots, \gamma_t$ and for every very reduced word $w = g_l^{n_l} \cdots g_1^{n_1}$ with $g_j$ in~$E_\Gamma$ and $n_j \geq 1$, we have
\[\lambda(w) - \sum_{1 \leq j \leq l} n_j \lambda(g_j) \in l \cdot M_\eps
\quad\text{and}\quad
\mu(w) - \sum_{1 \leq j \leq l} n_j \lambda(g_j) \in (l+1) \cdot M_\eps.\]
\end{proposition*}
\begin{remark*}
We have used additive notation for addition in~$A^\bullet$, even though it extends multiplication in~$A^o$ for which we had used multiplicative notation!
\end{remark*}
\begin{proof}
The morphisms $a \mapsto |\omega_i(a)|$ from $A^+$ to $(0, \infty)$ can be uniquely extended to continuous morphisms from the group~$A^\bullet$ to the multiplicative group~$(0, \infty)$, that we shall denote by~$\theta_i$.

We have for every $g$ in~$G$
\[\theta_i(\lambda(g)) = \lambda_1(\rho_i(g)).\]
Let us denote by~$C_{\omega_i}$ the constants introduced in \ref{sec:2.4} and let $C := \sup_{1 \leq i \leq r} C_{\omega_i}$. We then have, for every $g$ in~$G$ and~$i = 1, \ldots, r$,
\[C^{-1}\|\rho_i(g)\| \leq \theta_i(\mu(g)) \leq C\|\rho_i(g)\|.\]
Let~$C_\eps$~be the constant introduced in~\ref{sec:6.4} for a $k$-vector space with larger dimension than all of the~$V_i$, let $C'_\eps := C C_\eps$ and let us introduce the compact subset of~$A^\bullet$
\[M_\eps := \setsuch{a \in A^\bullet}{C'^{-1}_\eps \leq \theta_i(a) \leq C'_\eps \quad \forall i = 1, \ldots, r}.\]
Our statement is then a consequence of the following upper bounds given by Lemma~\ref{sec:2.4} and Proposition~\ref{sec:6.4}:
\[\theta_i \left( \lambda(w) - \sum_{1 \leq j \leq l} n_j \lambda(g_j) \right) = \frac{\lambda_1(\rho_i(w))}{\prod_{1 \leq j \leq l} \lambda_1(\rho_i(g_j))^{n_j}} \in [C_\eps^{-l},\; C_\eps^l] \quad\text{and}\]
\[\theta_i \left( \mu(w) - \sum_{1 \leq j \leq l} n_j \lambda(g_j) \right) = \frac{\theta_i(\mu(w))}{\|\rho_i(w)\|} \cdot \frac{\|\rho_i(w)\|}{\prod_{1 \leq j \leq l} \lambda_1(\rho_i(g_j))^{n_j}} \in [C^{-1}C_\eps^{-l-1},\; C C_\eps^{l+1}]. \qedhere\]
\end{proof}

\begin{corollary*}
Let $\Gamma$ be an $\eps$-Schottky subsemigroup (resp. subgroup) of~$G$ with generators $\gamma_1, \ldots, \gamma_t$, and let $\Omega^\bullet$~be an open convex cone of~$A^\bullet$ containing the half-lines generated by $\lambda(\gamma_1), \ldots, \lambda(\gamma_t)$ (resp. $\lambda(\gamma_1), \ldots, \lambda(\gamma_t), \lambda(\gamma_1^{-1}), \ldots, \lambda(\gamma_t^{-1})$).

Then there exists $m_0 \geq 1$ such that, for all $m \geq m_0$, the $\eps$-Schottky subsemigroup (resp. subgroup) $\Gamma_m$ of~$G$ with generators $\gamma_1^m, \ldots, \gamma_t^m$ satisfies $\mu(\Gamma_m) \subset \Omega \cup \{1\}$.
\end{corollary*}

\begin{proof}
Let us deal with the case where $\Gamma$~is a group (the case of a semigroup is easier). Let us first of all introduce, using Proposition~\ref{sec:5.1}, a compact subset~$M$ of~$A^\bullet$ such that, for every $w$ in $G$,
\[\mu \left( \{ w, w\gamma_1^{-1}, w\gamma_2^{-1} \} \right) \subset \mu(w) + M.\]

Let $g$~be an element of~$\Gamma_m$. We express it as a reduced word $g = g_l^m \cdots g_1^m$. One of the three words $w = g$, $g\gamma_1$ or~$g\gamma_2$ is very reduced. Hence we can apply the previous proposition to that word. Let $M' := \{0, \lambda(\gamma_1), \lambda(\gamma_2)\}$. We then have
\begin{align*}
\mu(g) \in \mu(w) + M
&\subset \left(\sum_{1 \leq j \leq l} m \lambda(g_j)\right) + M + M' + (l+2)M_\eps \\
&\subset \sum_{1 \leq j \leq l}(m \lambda(g_j) + M''),
\end{align*}
where $M''$~is some convex compact subset of~$A^\bullet$ containing both~$0$ and~$M + M' + 3M_\eps$. Hence it suffices to take~$m$ large enough for $\lambda(g_j) + \frac{1}{m}M''$ to be contained in~$\Omega^\bullet$.
\end{proof}

\begin{remark*}
Let $M_\Gamma$ and~$\Lambda_\Gamma$ denote the smallest closed convex cones in~$A^\bullet$ containing respectively $\mu(\Gamma)$ and~$\lambda(\Gamma)$. This calculation also proves that as $m$ tends to infinity, $M_{\Gamma_m}$ and~$\Lambda_{\Gamma_m}$ converge, in the sense of the Hausdorff distance in the projective space corresponding to~$A^\bullet$, to the convex hull of the half-lines $\Lambda_{\gamma_1}, \ldots, \Lambda_{\gamma_t}$ (resp. $\Lambda_{\gamma_1}, \ldots, \Lambda_{\gamma_t}, \iota(\Lambda_{\gamma_1}), \ldots, \iota(\Lambda_{\gamma_t})$).
\end{remark*}

\subsection{Construction of the group~$\Gamma$}
\label{sec:7.4}
\begin{theorem*}
Let $k$ be a local field, $\mathbf{G}$ a simply-connected isotropic semisimple $k$-group, $G := \mathbf{G}_k$, $\mu: G \to A^+$ some Cartan projection and $\iota: A^+ \to A^+$ the opposition involution.
\begin{enumerate}[label=\alph*)]
\item Let $\Omega$ be a nonempty convex open cone in~$A^+$ (see \ref{sec:7.1}). Then there exists a discrete subsemigroup~$\Gamma$, Zariski-dense in~$\mathbf{G}$, such that $\mu(\Gamma) \subset \Omega \cup \{1\}$.
\item If additionally $\iota(\Omega) = \Omega$, then there exists a discrete free subgroup~$\Gamma$, Zariski-dense in~$G$, such that $\mu(\Gamma) \subset \Omega \cup \{1\}$.
\end{enumerate}
\end{theorem*}

\begin{proof}
Let us choose some points $a_j$ in~$\Omega \cap A^{++}$. We start by constructing elements $\gamma_1, \ldots, \gamma_t$ of~$G$ as in Lemma~\ref{sec:7.2}. We then have $\lambda(\gamma_j) = a_j \in \Omega$ and, since $\iota(\Omega) = \Omega$, we have $\lambda(\gamma_j^{-1}) = \iota(a_j) \in \Omega$. Then Lemma~\ref{sec:7.2} and Corollary~\ref{sec:7.3} imply that there exists $m \geq 1$ such that the semigroup (resp. group) $\Gamma_m$ with generators $\gamma_1^m, \ldots, \gamma_t^m$ is $\eps$-Schottky hence in particular discrete and free, that it is Zariski-dense and that $\mu(\Gamma) \subset \Omega \cup \{1\}$.
\end{proof}

\subsection{Criterion for existence of a free subgroup acting properly on~$G/H$}
\label{sec:7.5}

\begin{theorem*}
Let $k$~be a local field, $\mathbf{G}$ a semisimple $k$-group, $\mathbf{H}$ a reductive $k$-subgroup of~$\mathbf{G}$, $\mathbf{A}_\mathbf{H}$ a maximal $k$-split torus of~$\mathbf{H}$, $\mathbf{A}$ a maximal $k$-split torus of~$\mathbf{G}$ containing~$\mathbf{A}_\mathbf{H}$, $G$,~$H$, $A_H$, $A$ the $k$-points, $W$~the Weyl group of~$G$ in~$A$, $A^+$~a positive Weyl chamber and $B^+$ the subset of~$A^+$ formed by the fixed points of the opposition involution (see~\ref{sec:2.2}).

There exists a discrete, non virtually unipotent subgroup~$\Gamma$ acting properly on~$G/H$ if and only if for every $w$ in~$W$, $w A_H$ does not contain~$B^+$.

In this case, we can always choose~$\Gamma$ to be free and Zariski-dense in~$G$.
\end{theorem*}
\begin{remark*}
When $\characteristic(k) = 0$, we can replace the ``non virtually nilpotent'' condition by the ``non virtually abelian'' condition.
\end{remark*}

\begin{proof}
By (\cite{Mar2} I.1.5.5 and I.2.3.1) we may assume that $\mathbf{G}$ is simply connected. We have the equality
\[\mu(H) = \mu(A_H) = \bigcup_{w \in W} (w A_H \cap A^+).\]

If there exists $w$ in~$W$ such that $w A_H$ contains~$B^+$, then $\mu(H)$ also contains~$B^+$ and the conclusion follows from Theorem~\ref{sec:3.3} and Corollary~\ref{sec:4.1}.

Otherwise, we can find a convex open cone~$\Omega^\bullet$ in~$A^\bullet$ that is invariant by~$\iota$ and whose closure has trivial intersection with each of the $\mathbb{R}$-vector subspaces of~$A^\bullet$ generated by some $w A_H \cap A^o$. We set $\Omega := \Omega^\bullet \cap A^+$. We then choose a discrete free subgroup~$\Gamma$ Zariski-dense in~$G$ such that $\mu(\Gamma) \subset \Omega \cap \{1\}$ (Theorem~\ref{sec:7.4}). Corollary~\ref{sec:5.2} then proves that this group~$\Gamma$ acts properly on~$G/H$.
\end{proof}

For semigroups, the same proof furnishes the following result.

\begin{proposition*}
Let $\mathbf{G}$ be a semisimple $k$-group, $\mathbf{H}$ a reductive $k$-subgroup of~$\mathbf{G}$, $G$ and $H$ the $k$-points. Suppose that $\rank_k(\mathbf{H}) \neq \rank_k(\mathbf{G})$.

Then there exists a Zariski-dense discrete free subsemigroup of~$G$ that acts properly on~$G/H$.
\end{proposition*}

\begin{remark*} (\cite{Kob1}) When $\rank_k(\mathbf{H}) = \rank_k(\mathbf{G})$, the only discrete subsemigroups acting properly on~$G/H$ are finite.
\end{remark*}

\subsection{Proof of the corollaries in the introduction}
\label{sec:7.6}

Corollary~\ref{corollary_1.2.1} is a particular case of the following corollary.
\begin{corollary*} ($\characteristic k = 0$)
We keep the notations of Theorem~\ref{sec:7.5}.

If there exists $w$ in $W$ such that $w A_H$ contains $B^+$, then $G/H$ has no compact quotient.
\end{corollary*}
\begin{proof}
This follows from Theorem~\ref{sec:7.5} and Corollary~\ref{sec:4.1}.
\end{proof}

To prove Corollary~\ref{corollary_1.2.2}, it suffices to verify, by (\cite{Kob2}~1.9), that the following examples have no compact quotients:
\begin{align*}
G_\mathbb{C}/H_\mathbb{C} =
&\SO(4n+2, \mathbb{C})/\SO(4n+1, \mathbb{C})\quad (n \geq 1), \\
&\SL(2n, \mathbb{C})/\Sp(n, \mathbb{C})\quad (n \geq 2) \text{ and} \\
&E_{6, \mathbb{C}}/F_{4, \mathbb{C}}.
\end{align*}

To prove Corollary~\ref{corollary_1.2.3}, we must verify that
\[G/H = \SO(2n+1, 2n+1)/\SO(2n, 2n+1) \quad (n \geq 1)\]
has no compact quotient.

In each of these cases, we verify that for a suitable choice of~$A^+$, the set~$B^+$ is contained in~$H$; and we apply the previous corollary.

\noindent {\scshape Yale University Mathematics Department, PO Box 208283, New Haven, CT 06520-8283, USA} \\
\noindent {\itshape E-mail address:} \url{ilia.smilga@normalesup.org}
\end{document}